\documentclass[12pt, reqno]{amsart}
\usepackage{latexsym, url}
\usepackage{amsthm}
\usepackage{amsmath}
\usepackage{amsfonts}
\usepackage{amssymb}
\usepackage[hidelinks]{hyperref}
\usepackage[dvips]{graphicx}
\usepackage{color}
\usepackage[all,pdf,cmtip]{xy}
\usepackage{wrapfig}

\newtheoremstyle{repeat}{}{}{\itshape}{}{\bfseries}{.}{.5em}{#3, repeated}

\newtheorem{theo}{Theorem}[section]
\newtheorem{lemma}[theo]{Lemma}

\newtheorem{propo}[theo]{Proposition}
\newtheorem{coro}[theo]{Corollary}
\newtheorem{fact}[theo]{Fact}

\theoremstyle{definition}
\newtheorem{question}[theo]{Question}
\newtheorem{exam}[theo]{Example}
\newtheorem{exams}[theo]{Examples}

\newtheorem{rem}[theo]{Remark}
\newtheorem{defi}[theo]{Definition}

\theoremstyle{repeat}
\newtheorem*{repeated-theorem}{Repeat}

\newcommand\Ind{\operatorname{Ind}}

\newcommand\Mod{\operatorname{\bf Mod}}

\newcommand\op{\operatorname{op}}

\newcommand\Set{\operatorname{\bf Set}}
\newcommand\Ab{\operatorname{\bf Ab}}

\newcommand\Bil{\operatorname{\bf Bil}}
\renewcommand\Vec{\operatorname{\bf Vec}}
\newcommand\BinFunc{\operatorname{\bf BinFunc}}

\newcommand\ACF{\operatorname{\bf ACF}}
\newcommand\EAF{\operatorname{\bf EAF}}

\newcommand\Gra{\operatorname{\bf Gra}}

\newcommand\Coalg{\operatorname{\bf Coalg}}

\newcommand\linspan{\operatorname{span}}
\newcommand\lin{\operatorname{lin}}
\newcommand\Mono{\operatorname{Mono}}
\newcommand\Reg{\operatorname{Reg}}
\newcommand\Emb{\operatorname{Emb}}

\newcommand\td{\operatorname{td}}
\newcommand\alg{\operatorname{alg}}

\newcommand\ca{\mathcal {A}}

\newcommand\cc{\mathcal {C}}
\newcommand\cd{\mathcal {D}}

\newcommand\ce{\mathcal {E}}

\newcommand\cl{\mathcal {L}}
\newcommand\cm{\mathcal {M}}

 \newbox\noforkbox \newdimen\forklinewidth

\def\Ind#1#2{#1\setbox0=\hbox{$#1x$}\kern\wd0\hbox to 0pt{\hss$#1\mid$\hss}
\lower.9\ht0\hbox to 0pt{\hss$#1\smile$\hss}\kern\wd0}
\def\ind{\mathop{\mathpalette\Ind{}}}
\def\Notind#1#2{#1\setbox0=\hbox{$#1x$}\kern\wd0\hbox to 0pt{\mathchardef
\nn="3236\hss$#1\nn$\kern1.4\wd0\hss}\hbox to 0pt{\hss$#1\mid$\hss}\lower.9\ht0
\hbox to 0pt{\hss$#1\smile$\hss}\kern\wd0}

\newcommand{\nf}{\ind}

\title{Lifting independence along functors}
\author[M. Kamsma and J. Rosick\'{y}]{M. Kamsma and J. Rosick\'{y}}
\address{
\newline M. Kamsma\newline
Department of Mathematics and Statistics\newline
Masaryk University, Faculty of Sciences\newline
Kotl\'{a}\v{r}sk\'{a} 2, 611 37 Brno, Czech Republic\newline
mark@markkamsma.nl\newline
\newline J. Rosick\'{y}\newline
Department of Mathematics and Statistics\newline
Masaryk University, Faculty of Sciences\newline
Kotl\'{a}\v{r}sk\'{a} 2, 611 37 Brno, Czech Republic\newline
rosicky@math.muni.cz
}

\date{1 September 2025}
 
\begin{document}

\begin{abstract}
Given a functor $F: \cc \to \cd$ and a model-theoretic independence relation on $\cd$, we can lift that independence relation along $F$ to $\cc$ by declaring a commuting square in $\cc$ to be independent if its image under $F$ is independent. For each property of interest that an independence relation can have we give assumptions on the functor that guarantee the property to be lifted.
\end{abstract}

\maketitle
\begin{center}
Dedicated to Robert Par\'{e} for his 80th birthday.
\end{center}

\tableofcontents

\section{Introduction}
Frequently in mathematics the notion of independence comes up. For example: linear independence, algebraic independence, or probabilistic independence. In model theory, specifically in stability theory, independence plays a central role. Here it takes the form of the abstract notion of forking as developed by Shelah \cite{Sh}, and the aforementioned examples are specific instances of this general theory. Initially, forking was developed for the very well-behaved class of stable theories. In the late nineties Kim and Pillay generalised this work to the class of simple theories, a class that strictly contains the class of stable theories \cite{KP}. Much more recently, Kaplan and Ramsey defined a notion of Kim-forking, allowing for a generalisation to the class of NSOP$_1$ theories \cite{KapRam}, which strictly contains the simple theories. 

While the notions of forking and Kim-forking are heavily syntactic, the independence relation that they yield allows for purely semantic treatment. This goes back to Harnik and Harrington's work \cite{HH} on independence in stable theories, where it is proved that a theory is stable if and only if an independence relation satisfying a certain list of properties exists and that said independence relation is unique. Later, si\-mi\-lar fully semantic treatments of independence were given in simple theories, by Kim and Pillay \cite{KP}, and in NSOP$_1$ theories, by Dobrowolski and the first named author \cite{DK}. Again, we get that a theory is simple (resp.\ NSOP$_1$) if and only if there is an independence relation satisfying one less property (resp.\ two less properties) compared to the stable case, and this independence relation is unique. Such a unique independence relation is called the canonical independence relation.

Nowadays, we use Makkai's notation \cite{M} for independence relations, namely the $\ind$ symbol. For example, we can define the independence relation $\ind^{\lin}$ given by linear independence in vector spaces as follows. If $A, B, C$ are subsets of a vector space $V$, then we write
\[
A \ind^{\lin, V}_C B \quad \Longleftrightarrow \linspan(A \cup C) \cap \linspan(B \cup C) = \linspan(C).
\]
This is indeed the canonical independence relation for vector spaces. That is, the unique independence relation witnessing that the theory is stable. Generally, $A \ind^V_C B$ should be read as ``$A$ is independent from $B$ over $C$ in $V$'', and can intuitively be understood as ``all the information that $B$ has about $A$ is already contained in $C$''.

The semantic treatment of independence relations opened new ways of studying them. In particular, a category-theoretic treatment became possible. As the original motivation for accessible categories was to develop a ``categorical model theory'' \cite{MP}, it should be no surprise that it is the framework of accessible categories in which the categorical treatment of independence relations takes place. In the combined work of \cite{LRV, K, K1}, canonicity of categorical independence is proved, which roughly states that in an accessible category there can be only one independence relation with the properties that independence is known to have in NSOP$_1$, simple or stable theories (see also Fact \ref{canonicity}). These canonicity results tell us what the correct definitions for categorical independence are, allowing for a new categorical approach to independence. For example, in \cite{LRV1} a categorical construction of stable independence is given, and in \cite{KR} the same is done for (potentially unstable) simple and NSOP$_1$-like independence. In this context, it should also be mentioned that substantial work on independence has been done in the context of AECs (e.g., \cite{ShAEC, V, HytKes, BGKV}), which is a different framework from the categorical one, but allows for a direct translation \cite{BR}.

In this paper we contribute to the work on categorical independence by providing various criteria for which properties of an independence relation lift along a functor. This is best illustrated with an example (which is essentially the well-known model-theoretic example of the random graph). Let $F: \Gra_{\Emb} \to \Set_{\Mono}$ be the forgetful functor from graphs with embeddings to sets with monomorphisms. There is a stable independence relation $\ind$ on $\Set_{\Mono}$, which yields an independence relation $F^{-1}(\ind)$ on $\Gra_{\Emb}$ by declaring a commuting square to be independent if its image under $F$ is $\ind$-independent. The question then becomes: what properties of $\ind$ lift to $F^{-1}(\ind)$? In this case the answer will be: all the properties of a stable independence relation, except for one (uniqueness).

\textbf{Main results.} For each property of interest that an independence relation can have we give assumptions on the functor $F$ that guarantee the property to be lifted. For a good number of properties (uniqueness, existence and 3-amalgamation) there are two flavours of assumptions: we can either assume $F$ to be a left multiadjoint, or we can assume something about the image of $F$, like a higher dimensional variant of cofinality. Here cofinality is a property of a functor, which generalises the usual cofinality of a subset of a poset. We then summarise everything in Theorems \ref{lift-everything}, \ref{left-multiadjoint-lifts-everything} and \ref{left-multiadjoint-lifts-simplicity}, which tell us when a stable, simple or NSOP$_1$-like independence relation lifts to an independence relation of similar strength.

\textbf{Overview.} The paper is built up as follows. In Section \ref{sec:prelims} we start with the preliminaries, which mainly consist of the definition of a categorical independence relation. In Section \ref{sec:lifting-independence} we give a precise definition of lifting independence and establish the first results of properties that lift, most notably the lifting of the union and accessibility properties. In the sections after that (\ref{sec:lifting-uniqueness}, \ref{sec:lifting-existence}, \ref{sec:lifting-3-amalgamation} and \ref{sec:lifting-base-monotonicity}) we discuss one property per section and give conditions for when it lifts. Section \ref{sec:strong-3-amalgamation} is an exception, there we discuss a strengthening of the 3-amalgamation property, and show how for simple (and thus in stable) independence relations this stronger property follows from the rest of the properties. Finally, in Section \ref{sec:lifting-everything} we summarise everything in three main theorems.

\section{Preliminaries}
\label{sec:prelims}
We are working in the framework of accessible categories. We assume the reader is familiar with this. Good references for this topic are \cite{AR,MP}.
\begin{defi}
\label{independence-relation}
An \emph{independence relation} $\nf$ on a category $\cc$ is a class of commuting squares.
\[
\xymatrix@=1pc{
  A \ar[r] & M  \\
  C \ar[r] \ar[u] & B \ar[u] &
}
\]
If a commuting square is in the relation, we call it \emph{independent} and write $A \nf_C^M B$.
\end{defi}
We compare this definition to the example of linear independence in the introduction. If we work in the category of vector spaces (over some fixed field) and injective linear maps then a commuting square like the above corresponds to three subspaces $A, B, C \subseteq M$ such that $C \subseteq A \cap B$. So the only difference with the independence relation in the introduction is that $A$, $B$ and $C$ are now subspaces and that $C$ is contained in $A$ and $B$, instead of taking three arbitrary subsets. However, this difference is not substantial and both approaches can be recovered from one another, where in the categorical approach we replace ``arbitrary subsets'' by ``arbitrary subobjects'', see \cite[Remarks 2.10 and 2.11]{KR} for more details.\footnote{The reference deals with subobjects in a fixed category $\cc$, so in our example these would still be subspaces rather than arbitrary subsets. However, the framework of AECats in \cite{K} was set up in a way to also deal with arbitrary subsets.}
\begin{defi}
We list properties of an independence relation:
\begin{description}
\item[Invariance\footnotemark]\footnotetext{When working in the model-theoretically traditional context of a monster model, the ``invariance'' property corresponds to the property that the independence relation is invariant under automorphisms of the monster model, hence the name.} In any commuting diagram like below we have $A \nf_C^M B$ if and only if $A \nf_C^N B$.
\[
\xymatrix@=1pc{
  A \ar[r] & M \ar[r] & N \\
  C \ar[r] \ar[u] & B \ar[u] &
}
\]
\item[Isomorphism] Given two isomorphic commuting squares like below, the inner square is independent if and only if the outer square is independent.
\[
\xymatrix@=1pc{
  A' \ar[rrr] \ar@{=}[dr]|\cong & & & M' \\
  & A \ar[r] & M \ar@{=}[ur]|\cong & \\
  & C \ar[u] \ar[r] & B\ar@{=}[dr]|\cong \ar[u] & \\
  C' \ar[uuu]\ar@{=}[ur]|\cong \ar[rrr] & & & B' \ar[uuu]
}
\]
\item[Monotonicity] In any commuting diagram like below $A \nf_C^M B$ implies $A \nf_C^M B'$.
\[
\xymatrix@=1pc{
  A \ar[rr] & & M \\
  C \ar[r] \ar[u] & B' \ar[r] & B \ar[u]
}
\]
\item[Transitivity] Independent squares can be composed, so in any commuting diagram like below we have that if the two squares are independent then the outer rectangle is independent.
\[
\xymatrix@=1pc{
  A \ar[r] & M \ar[r] & N \\
  C \ar[r] \ar[u] & B \ar[r] \ar[u] & D \ar[u]
}
\]
\item[Symmetry] We have $A \nf_C^M B$ if and only if $B \nf_C^M A$.
\item[Basic existence] Any commuting square with the bottom or left morphism an isomorphism is independent. 
\item[Existence] Any span can be completed to an independent square.
\item[Base Monotonicity\footnotemark]\footnotetext{Base monotonicity might be the hardest to compare to the usual model-theoretic formulation, because we have to work with commuting squares. The usual model-theoretic formulation, where $A, B, C, D$ are just subsets of some model $M$, says that if $A \ind^M_C D$ and $C \subseteq B \subseteq D$ then $A \ind^M_B D$, which is then equivalent to $A \cup B \ind^M_B D$ (modulo some basic properties). In the category-theoretic formulation $A'$ plays the role of $A \cup B$, and $A'$ may only exist in a bigger model $N$.} Given a commuting diagram consisting of the solid arrows below such that $A \nf_C^M D$, there are $A'$ and $N$ and the dashed arrows such that everything commutes and $A' \nf_B^N D$.
\[
\xymatrix@=2pc{
  & A' \ar@{-->}[r] & N \\
  A \ar@{-->}[ur] \ar[rr] & & M \ar@{-->}[u] \\
  C \ar[r] \ar[u] & B \ar[r] \ar@{-->}[uu] & D \ar[u]
}
\]
\item[Uniqueness] Given a commuting diagram consisting of the solid arrows below such that both $A \nf_C^M B$ and $A \nf_C^{M'} B$, there are $N$ and the dashed arrows such that everything commutes.
\[
\xymatrix@=2pc{
  & M \ar@{-->}[r] & N \\
  A \ar[rr] \ar[ur] & & M' \ar@{-->}[u] \\
  C \ar[u] \ar[r] & B \ar[uu] \ar[ur] &
}
\]
\item[$3$-amalgamation] Given a commuting diagram consisting of the solid arrows below (we call this a \emph{horn}) with every square independent (we call this an \emph{$\nf$-independent horn}), there are $N$ and the dashed arrows such that everything commutes and $A \nf_M^N N_3$.
\[
\xymatrix@=2pc{
        & N_2 \ar @{-->}[rr] & & N\\
        A \ar[rr] \ar[ur] &
        & N_1 \ar @{-->}[ur] & \\
        & C \ar'[u][uu] \ar'[r][rr] & & N_3 \ar@{-->} [uu] & \\
        M \ar[ur]\ar[rr]\ar[uu] & & B \ar[ur]\ar[uu] &  
      }
\]
\end{description}
\end{defi}
Note that isomorphism follows from basic existence and transitivity. Also, basic existence follows from existence and invariance \cite[Lemma 3.12]{LRV}.

Given a category $\cc$ we write $\cc^2$ for the category that has as objects morphisms from $\cc$ and as morphisms the commuting squares in $\cc$. For an independence relation $\nf$ on $\cc$ that satisfies basic existence and transitivity we then write $\cc_{\nf}$ for the subcategory of $\cc^2$ where the morphisms are independent squares.
\begin{defi}
We define two more properties for an independence relation $\nf$, which are properties of $\cc_{\nf}$ so we need to assume basic existence and transitivity.
\begin{description}
\item[Accessible] The category $\cc_{\nf}$ is accessible.
\item[Union] The category $\cc_{\nf}$ has directed colimits and these are preserved by the inclusion functor $\cc_{\nf} \hookrightarrow \cc^2$.
\end{description}
\end{defi}
We have now defined all the properties of independence relations that we will consider. As discussed in the introduction, each of the model-theoretic classes of stable, simple and NSOP$_1$ can be characterised by the list of properties that the canonical independence relation has. We now recall this characterising list of properties for each of these classes, and we will name an independence relation satisfying such a list after the class that it characterises.
\begin{defi}
\label{stable-simple-nsop1-independence}
An independence relation $\nf$ is called
\begin{description}
\item[NSOP$_1$-like] If it is accessible and satisfies invariance, monotonicity, transitivity, symmetry, existence, 3-amalgamation and union.
\item[Simple] If it is NSOP$_1$-like and satisfies  base monotonicity.
\item[Stable] if it is simple and satisfies uniqueness.
\end{description}
\end{defi}
We also recall a simplified version of the framework of \emph{AECats} (\emph{Abstract Elementary Categories}) from \cite{K}, as this is the categorical framework that we want to work in. In \cite{K} these are defined as a pair of categories $(\cc, \cm)$, but will only be interested in the case $\cc = \cm$, hence the simplification.
\begin{defi}
An \emph{AECat} is an accessible category with directed colimits and where all morphisms are monomorphisms.
\end{defi}
Normally, we are only interested in AECats with the amalgamation property (i.e., any span of morphisms can be completed to a commuting square). In the main results we will only deal with AECats that have an independence relation that satisfies existence, and so the amalgamation property is automatic.
\begin{fact}[{Canonicity of categorical independence, \cite{LRV, K, K1}}]
\label{canonicity}
Let $\cc$ be an AECat and let $\ind$ and $\ind'$ be independence relations on $\cc$, and suppose that $\ind$ is NSOP$_1$-like.
\begin{enumerate}
\item If $\ind'$ is simple then $\ind = \ind'$.
\item If $\cc$ satisfies the technical assumption ``existence axiom'' and $\ind'$ is NSOP$_1$-like then $\ind = \ind'$.
\end{enumerate}
\end{fact}
The interested reader can refer to \cite[Definition 6.14]{K1} for the definition of the ``existence axiom'', but it is not important for this paper. The point of the above fact is to recall the precise statement of canonicity of categorical independence, giving context to Definition \ref{stable-simple-nsop1-independence}.

We finish this section by recalling some notation.
\begin{defi}
\label{class-of-morphisms-notation}
Let $\cm$ be a class of morphisms in a category $\cc$.
\begin{itemize}
\item If $\cm$ is closed under composition and contains all isomorphisms we call it \emph{composable}.
\item If $gf \in \cm$ implies that $f \in \cm$ we call it \emph{left-cancellable}.
\item If $gf, g \in \cm$ implies that $f \in \cm$ we call it \emph{coherent}.
\end{itemize}
For a composable class of morphisms we write $\cc_\cm$ for the subcategory of $\cc$ whose objects are all of those in $\cc$ and whose morphisms are precisely those from $\cm$.
\begin{itemize}
\item We call $\cm$ \emph{continuous} if $\cc_\cm$ is closed under directed colimits in $\cc$.
\item We call $\cm$ \emph{accessible} if $\cc_\cm$ is accessible.
\end{itemize}
\end{defi}
An example of a left-cancellable composable class is the class $\Mono$ of all monomorphisms in $\cc$. Another example is the class $\Reg$ of regular monomorphisms in a coregular category. We will usually be interested in the case where $\cm$ consists of monomorphisms, by which we mean $\cm \subseteq \Mono$.

\section{Lifting independence}
\label{sec:lifting-independence}
We start this section by discussing two motivating examples that we will revisit throughout the paper. There will be more examples, but most often we will refer to these two.
\begin{exams}\label{motivating-examples}
These two examples are based on \cite{K2} and \cite{AHKK}.
\begin{enumerate}
\item Fix a field $K$ and let $\Bil_K$ be the category of bilinear spaces over $K$, whose morphisms are the injective linear maps that preserve the bilinear form. If the reader wishes, they may also restrict to symmetric or alternating bilinear spaces (and everything goes through in the same way). Let $\Vec_K$ be the category of vector spaces over $K$ with injective linear maps. In both categories, the canonical independence relation $\ind^{\lin}$ is given by linear independence. As discussed right after Definition \ref{independence-relation}, we can simplify the definition of linear independence from the introduction by only considering subspaces $A, B, C$ of some bilinear/vector space $V$ with $C \subseteq A \cap B$---i.e., those forming a commuting square in $\Vec_K$ or $\Bil_K$. We then have $A \ind^{\lin, V}_C B$ if and only if $A \cap B = C$. This relation is stable on $\Vec_K$ and is simple\footnote{There is a well-known model-theoretic example that bilinear spaces over infinite fields are never simple, when considered as a two-sorted structure (they can be NSOP$_1$, e.g., when the field is algebraically closed). We emphasise that in our example the field is fixed, which is exactly what gives us simplicity. For more details, we refer to \cite{K2}.} unstable on $\Bil_K$. Writing $F: \Bil_K \to \Vec_K$ for the obvious forgetful functor, we thus see that a commuting square in $\Bil_K$ is independent if and only if its image under $F$ is independent in $\Vec_K$.
\item An exponential field is a field $F$ (of characteristic 0) together with a group homomorphism $\exp: F^+ \to F^{\times}$ from the additive group to the multiplicative group. Such an exponential field is called an \emph{EA-field} if it is algebraically closed as a field.

Recall that for algebraically closed fields the canonical independence relation $\ind^{\td}$ is given by algebraic independence. That is, for subsets $A, B, C$ of an algebraically closed field $F$ we have $A \ind^{\td, F}_C B$ if for every finite tuple $\bar{a}$ from $A$ we have $\td(\bar{a}/C) = \td(\bar{a}/B \cup C)$, where $\td$ is the transcendence degree. This independence relation is stable. Given an EA-field $F$ and a subset $A \subseteq F$ we write $\langle A \rangle^{EA}_F$ for the smallest EA-subfield of $F$ containing $A$. The canonical independence relation $\ind^{EA}$ for EA-fields is then given by $A \ind^{EA, F}_C B$ if and only if $\langle A \cup C \rangle^{EA}_F \ind^{\td, F}_{\langle C \rangle^{EA}_F} \langle B \cup C \rangle^{EA}_F$ for arbitrary subsets $A, B, C \subseteq F$. So after applying a closure operation, we just look at the usual algebraic independence. The independence relation on EA-fields is NSOP$_1$-like non-simple.

We consider the following two categories: $\ACF$ is the category of algebraically closed fields of characteristic 0 and $\EAF$ is the category of EA-fields. The morphisms are simply embeddings in both cases. Writing $F: \EAF \to \ACF$ for the underlying field functor, we can reformulate the above to saying that a commuting square in $\EAF$ is independent if and only if its image under $F$ is independent in $\ACF$. In this formulation we only look at squares in $\EAF$, and therefore the step where we would have to take the $\langle - \rangle^{EA}$-closure trivialises. That is, for EA-subfields $A, B, C \subseteq F$ with $C \subseteq A \cap B$ the independence relation on EA-fields simplifies to $A \ind^{EA, F}_C B$ if and only if $A \ind^{\td, F}_C B$.

We briefly note that instead of EA-fields, we could also look at ELA-fields where the kernel of the exponential map is of a fixed isomorphism type. An ELA-field is an EA-field whose exponential map is surjective. While these are interesting objects to study, they add little value as an example in this paper, so we will only consider EA-fields and refer the interested reader to \cite{AHKK}.
\end{enumerate}
\end{exams}
In each case we see that the independence relation on the more complicated objects can be computed in terms of the simpler objects. Viewed differently: the independence relation on the more complicated objects arises by lifting the independence relation on the simpler objects along a forgetful functor.
\begin{defi}\label{lifting-independence}
Given a functor $F: \cc \to \cd$ and an independence relation $\nf$ on $\cd$, we define an independence relation $F^{-1}(\nf)$ on $\cc$, the \emph{lift} of $\nf$, by declaring a commuting square in $\cc$ to be $F^{-1}(\nf)$-independent if and only if its image in $\cd$ is $\nf$-independent:
\[
\vcenter{\vbox{\xymatrix@=3pc{
  A \ar[r] & D \\
  C \ar[u] \ar@{}[ur]|-{F^{-1}(\nf)} \ar[r] & B \ar[u]
}}}
\quad \Longleftrightarrow \quad
\vcenter{\vbox{\xymatrix@=3pc{
  F(A) \ar[r] & F(D) \\
  F(C) \ar[u] \ar@{}[ur]|-{\nf} \ar[r] & F(B) \ar[u]
}}}
\]
\end{defi}
\begin{propo}\label{lifting-basic-properties}
Let $F: \cc \to \cd$ be a functor and $\nf$ an independence relation on $\cd$. Then the following properties are lifted to $F^{-1}(\nf)$ on $\cc$: invariance, monotonicity, transitivity, symmetry and basic existence. That is, if $\nf$ has one of these properties then $F^{-1}(\nf)$ has that property.
\end{propo}
\begin{proof}
Each of these properties follows easily and quickly from the definitions.
\end{proof}
\begin{theo}\label{lifting-union-accessibility}
Let $\cc$ and $\cd$ be accessible categories having directed colimits and let $F: \cc \to \cd$ be a directed colimit preserving functor. Suppose that $\nf$ is an independence relation on $\cd$ that satisfies basic existence and transitivity.
\begin{enumerate}
\item If $\nf$ satisfies union then so does $F^{-1}(\nf)$.
\item If $\nf$ satisfies union and is accessible then the same holds for $F^{-1}(\nf)$.
\end{enumerate}
\end{theo}
\begin{proof}
The functor $F$ induces a directed colimit preserving functor $F^2: \cc^2 \to \cd^2$. We thus have a pullback
\[
\xymatrix@=2pc{
  \cc^2 \ar[r]^{F^2} & \cd^2 \\
  \cc_{F^{-1}(\nf)} \ar@{^{(}->}[u] \ar@{}[ur] \ar[r] & \cd_{\nf} \ar@{^{(}->}[u]
}
\]
By basic existence, the inclusion functor $\cd_{\nf} \hookrightarrow \cd^2$ is an isofibration. Therefore, the pullback is a pseudopullback (see, e.g., \cite[Theorem 1]{JS}).

The 2-category of categories with directed colimits and directed colimit preserving functors is closed under pseudopullbacks. So if $\nf$ satisfies union then the cospan $\cc^2 \xrightarrow{F^2} \cd^2 \hookleftarrow \cd_{\nf}$ lives in this 2-category, so the entire diagram lives in this 2-category. We conclude that $F^{-1}(\nf)$ has union.

If $\nf$ is in addition also accessible then we can use the fact that accessible categories are closed under pseudopullbacks (see \cite[Theorem 5.1.6]{MP} or \cite[Exercise 2.n]{AR}) to conclude that $F^{-1}(\nf)$ is accessible.
\end{proof}
Proposition \ref{lifting-basic-properties} and Theorem \ref{lifting-union-accessibility} apply to both forgetful functors in Examples \ref{motivating-examples}.

\section{Lifting uniqueness}
\label{sec:lifting-uniqueness}
\begin{defi}\label{amalgamation-of-squares}
We say that two commuting squares that share the same base span, such as the solid arrows in the diagram below, can be \emph{amalgamated} if there are $E$ and the dotted arrows as in the diagram below that make everything commute.
\[
\xymatrix@=1.8pc{
    & D \ar@{{}*{\cdot}>}[r] & E \\
    B \ar[ru] \ar[rr] & & D' \ar@{{}*{\cdot\!}>}[u] \\
    A \ar[u] \ar[r] & C \ar[uu]  \ar[ur] & \\
}
\]
\end{defi}
In this terminology, the uniqueness property for an independence relation $\nf$ can be reformulated by saying that any two $\nf$-independent squares that share the same base span can be amalgamated.
\begin{defi}\label{reflecting-amalgamation-of-squares}
A functor $F: \cc \to \cd$ is said to \emph{reflect amalgamation of squares} if whenever we have two squares in $\cc$ that share the same base span and their images under $F$ can be amalgamated in $\cd$ then the original squares can already be amalgamated in $\cc$.
\[
\vcenter{\vbox{\xymatrix@=1.5pc{
    & F(D) \ar@{{}*{\cdot}>}[r] & E \\
    F(B) \ar[ru] \ar[rr] & & F(D') \ar@{{}*{\cdot\!}>}[u] \\
    F(A) \ar[u] \ar[r] & F(C) \ar[uu]  \ar[ur] & \\
}}}
\quad \rightsquigarrow \quad
\vcenter{\vbox{\xymatrix@=1.5pc{
    & D \ar@{{}*{\cdot}>}[r] & H \\
    B \ar[ru] \ar[rr] & & D' \ar@{{}*{\cdot\!}>}[u] \\
    A \ar[u] \ar[r] & C \ar[uu]  \ar[ur] & \\
}}}
\]
\end{defi}
The following is an immediate consequence of the definitions.
\begin{propo}\label{reflecting-amalgamation-of-squares-lifts-uniqueness}
Suppose that $F: \cc \to \cd$ is a functor that reflects amalgamation of squares and that $\nf$ is an independence relation on $\cd$ satisfying uniqueness, then $F^{-1}(\nf)$ satisfies uniqueness.
\end{propo}
\begin{exam}\label{bilinear-spaces-jointly-reflecting-amalgamation-of-squares}
In addition to the forgetful functor $F: \Bil_K \to \Vec_K$ from Example \ref{motivating-examples}(1), we also consider the following functor. Let $\BinFunc_K$ be the category whose objects are sets $X$ with a binary function $X \times X \to K$ into $K$, and the morphisms are injections of these sets that respect the binary function, then we take $G: \Bil_K \to \BinFunc_K$ to be the functor that forgets about the vector space structure. Neither $F$ nor $G$ reflects amalgamation of squares, but $\langle F, G \rangle: \Bil_K \to \Vec_K \times \BinFunc_K$ does.
\end{exam}
The situation in the above example where we really have to consider a product of categories happens often in practice, so it will be useful to give this a name.
\begin{defi}\label{jointly-reflecting-amalgamation-of-squares}
We say that a family of functors $\{F_j: \cc \to \cd_j\}_{j \in J}$ \emph{jointly reflect amalgamation of squares} if $\langle F_j \rangle_{j \in J}: \cc \to \prod_{j \in J} \cd_j$ reflects amalgamation of squares.
\end{defi}
Note that the above condition is further equivalent to saying that two squares in $\cc$ can be amalgamated if and only if their images under each $F_j$ can be amalgamated.
\begin{rem}\label{reflecting-amalgamation-of-squares-vs-types}
We compare the amalgamation of squares to the model-theoretic notion of (Galois) types (see e.g., \cite[Section 3]{K}). This remark is not used anywhere else in this paper, so the reader may choose to skip it.

We recall that two tuples of morphisms $(a_i: A_i \to D)_{i \in I}$ and $(a'_i: A_i \to D')_{i \in I}$ have the same Galois type if they can be amalgamated, in the sense that there are $D \xrightarrow{f} E \xleftarrow{g} D'$ such that $fa_i = ga'_i$ for all $i \in I$. If a category has the amalgamation property this is an equivalence relation. Being able to amalgamate two squares is then a special case of having the same Galois type.

In the model-theoretic context, having the same Galois type often admits a syntactic description. For example, if the category in question is the category $\Mod(T)$ of models of a first-order theory $T$ with elementary embeddings then two tuples of morphisms have the same Galois type if and only if we can enumerate their images such that these enumerations satisfy the same first-order formulas (such a set of formulas satisfied by a tuple is called a \emph{type}). This works similarly for categories of models in some other logics, such as positive logic \cite{BY} or continuous logic \cite{BYBHU}. Note that these categories always have the amalgamation property.

We can thus define a functor $F$ between categories with the amalgamation property to be \emph{Galois type reflecting} if two tuples of morphisms have the same Galois type if and only if their images under $F$ have the same Galois type. This is more general than reflecting amalgamation of squares, but is in practice sometimes easier to check. Example \ref{bilinear-spaces-jointly-reflecting-amalgamation-of-squares} can be cast in this language: take two tuples of monomorphisms of bilinear spaces $(a_i: A_i \to D)_{i \in I}$ and $(a'_i: A_i \to D')_{i \in I}$, and view each $a_i$ as a tuple enumerating its image in $D$ (similarly for the $a'_i$). Then these can be amalgamated in $\Bil_K$ if and only if:
\begin{enumerate}
\item $\operatorname{span}(a_i : i \in I) \cong \operatorname{span}(a'_i : i \in I)$, that is, $(a_i)_{i \in I}$ and $(a'_i)_{i \in I}$ satisfy the same $K$-linear equations;
\item $[x, y] = [x', y']$ for all $x,y$ singletons in the concatenated tuple $(a_i)_{i \in I}$ and $x',y'$ of matching positions in $(a'_i)_{i \in I}$, where $[-, -]$ is the $K$-bilinear form on $D$ (and on $D'$, by overloading notation).
\end{enumerate}
Model-theoretically we can say that the type of a tuple is determined by the type of that tuple in the underlying vector space together with the type of that tuple with respect to the bilinear form (seen as just a binary function).
\end{rem}
\begin{coro}\label{jointly-reflecting-amalgamation-of-squares-lifts-uniqueness}
Suppose that $\{F_j: \cc \to \cd_j\}_{j \in J}$ is a family of functors that jointly reflect amalgamation of squares and suppose that $\nf^j$ is an independence relation on each $\cd_j$ satisfying uniqueness. Let $\bigcap_{j \in J} F_j^{-1}(\nf^j)$ be the independence relation on $\cc$ where a square is independent if it is $F_j^{-1}(\nf^j)$-independent for all $j \in J$. Then $\bigcap_{j \in J} F_j^{-1}(\nf^j)$ satisfies uniqueness.
\end{coro}
\begin{proof}
Define an independence relation $\nf$ on $\prod_{j \in J} \cd_j$ by declaring a square independent if and only if its $j$th component is $\nf^j$-independent in $\cd_j$ for all $j \in J$. It is straightforward to check that $\nf$ satisfies uniqueness. Then $\bigcap_{j \in J} F_j^{-1}(\nf^j) = \langle F_j \rangle_{j \in J}^{-1}(\nf)$ and $\langle F_j \rangle_{j \in J}$ reflects amalgamation of squares, so the result follows from Proposition 
\ref{reflecting-amalgamation-of-squares-lifts-uniqueness}.
\end{proof}
\begin{exam}\label{bilinear-spaces-uniqueness-lifting}
We continue Example \ref{bilinear-spaces-jointly-reflecting-amalgamation-of-squares}. We define $\nf$ on $\BinFunc_K$  as follows. For a set $X$ with a binary function $f: X \times X \to K$, and subsets $A, B, C \subseteq X$ with $C \subseteq A \cap B$ we set $A \nf^X_C B$ if and only if $A \cap B = C$ and for any $a \in A \setminus C$ and $b \in B \setminus C$ we have that $f(a, b) = f(b, a) = 0$. It is well-known that $\nf^{\lin}$ on $\Vec_K$ satisfies uniqueness and it is easy to check that $\nf$ on $\BinFunc_K$ satisfies uniqueness. Applying Proposition 
\ref{jointly-reflecting-amalgamation-of-squares-lifts-uniqueness} to the forgetful functors $F: \Bil_K \to \Vec_K$ and $G: \Bil_K \to \BinFunc_K$ these independence relations lift to an independence relation $\nf^*$ on $\Bil_K$, which satisfies uniqueness. This can also be straightforwardly verified directly from the explicit description of $\nf^*$ on $\Bil_K$: for $A, B, C \subseteq V$ subspaces of a bilinear space $V$ with $C \subseteq A \cap B$ we have that $A \nf^{*,V}_C B$ if and only if $A \cap B = C$ and for any $a \in A \setminus C$ and $b \in B \setminus C$ we have that $[a, b] = [b, a] = 0$, where $[-, -]$ is the bilinear form on $V$. Note that this is different from the canonical independence relation on $\Bil_K$, which is $\ind^{\lin}$.
\end{exam}
We finish this section by considering a condition that implies reflection of amalgamation of squares, and much more. This condition is a weakening of being a left adjoint, which we recall here (see also \cite[Definition 4.24(1)]{AR}, dualised).
\begin{defi}\label{left-multiadjoint}
A functor $F: \cc \to \cd$ is called a \emph{left multiadjoint} if for every object $D$ in $\cd$ there is a family of morphisms $\{e_i: F(C_i) \to D\}_{i \in I}$ such that for every $e: F(C) \to D$ there is a unique $i \in I$ and unique $f: C \to C_i$ such that $e_i F(f) = e$.
\end{defi}
As mentioned before: being a left adjoint is a special case of being a left multiadjoint. If $F$ is left adjoint to $R: \cd \to \cc$ then for any $D$ in $\cd$ one can take the singleton $\{\varepsilon_D: FR(D) \to D\}$ to be the family of morphisms, where $\varepsilon: FR \to Id$ is the counit of the adjunction. Then given any $e: F(C) \to D$ the unique $f$ is given by $F(\tilde{e}): F(C) \to FR(D)$, where $\tilde{e}$ is the transpose of $e$ under the adjunction.
\begin{lemma}\label{left-multiadjoint-on-cocones}
Let $F: \cc \to \cd$ be a left multiadjoint, let $(X_i)_{i \in I}$ be a nonempty connected diagram in $\cc$ and let $(d_i: F(X_i) \to D)_{i \in I}$ be a cocone in $\cd$. Then there is a cocone $(f_i: X_i \to C)_{i \in I}$ in $\cc$ and $u: F(C) \to D$ such that $d_i = u F(f_i)$ for all $i \in I$.

Furthermore, for any other cocone $(f'_i: X_i \to C')_{i \in I}$ such that there is $u': F(C') \to D$ with $d_i = u' F(f'_i)$ for all $i \in I$ there is a unique $g: C' \to C$ such that $f_i = g f'_i$ for all $i \in I$.
\end{lemma}
\begin{proof}
Let $\{e_j: F(C_j) \to D\}_{j \in J}$ be the family of morphisms coming from the fact that $F$ is a left multiadjoint. For any $i \in I$ we let $j_i \in J$ be the unique index such that there is a unique $f_i: X_i \to C_{i_j}$ with $d_i = e_{j_i} F(f_i)$. We claim that $j_i = j_{i'}$ for all $i, i' \in I$. As the diagram is connected, any $i,i' \in I$ can be connected by a zig-zag of morphisms $i = i_1 \to i_2 \leftarrow i_3 \to \ldots \leftarrow i_n = i'$. Then for every $1 \leq k < n$ we have that one of $d_{i_k}$ and $d_{i_{k+1}}$ factors through the other. So by uniqueness $j_{i_k} = j_{i_{k+1}}$, and the claim follows.

We can thus take any $i^* \in I$ and we set $C = C_{j_{i^*}}$ and $u = e_{j_{i^*}}$, where nonemptiness of the diagram implies that there is at least one ${i^*} \in I$. Finally, uniqueness of the $f_i$'s then implies that $(f_i: X_i \to C)_{i \in I}$ is in fact a cocone.

For the furthermore part, we let $j' \in J$ be such that $u': F(C') \to D$ factors as $u' = e_{j'} F(g)$ for some $g: C' \to C_{j'}$. Since $d_{i^*}$ factors through $u'$ by assumption, we have by uniqueness that $j' = j_{i^*}$. The fact that $f_i = g f'_i$ for all $i \in I$ follows from the uniqueness of the $f_i$'s.
\end{proof}

\emph{Multicolimits} generalise colimits in the same way as left 
multiadjoints generalise left adjoints, i.e., the initial cocone is replaced by a multiinitial set of cocones (see, e.g., \cite[Definition 4.24]{AR}). The following corollary generalises \cite[Theorem 4.26(i)]{AR}.

\begin{coro}\label{multi}
Left multiadjoints preserve connected multicolimits.
\end{coro}

Often, an independence relation arises by starting with a locally presentable category and then passing to an accessible subcategory by restricting the morphisms. For example, \cite[Theorem 5.1]{LRV} states that if $\cc$ is a coregular locally presentable category with effective unions then $\cc_{\Reg}$ has a stable independence relation. The subcategory is generally no longer locally presentable, and thus less likely to be the domain or codomain of a left (multi)adjoint, because we lose completeness and cocompleteness of the categories and so theorems such as the adjoint functor theorem no longer apply.

\begin{propo}\label{left-multiadjoint-reflects-amalgamating-squares}
Let $F: \cc \to \cd$ be a left multiadjoint and let $\cm$ be a left-cancellable composable class of morphisms in $\cd$. Then the restriction $F: \cc_{F^{-1}(\cm)} \to \cd_\cm$ reflects amalgamation of squares.
\end{propo}
\begin{proof}
We start with two squares in $\cc_{F^{-1}(\cm)}$ that share their base span, and we assume that their images under $F$ can be amalgamated in $\cd_\cm$. Then, working in $\cd$ we find the commuting diagram below.
\[
\xymatrix@=2pc{
    & & & D \\
    & F(E_1) \ar@/^2pc/[rru] \ar@{{}*{\cdot}>}[r] & F(D') \ar@{-->}[ru] & \\
    F(A) \ar[rr] \ar[ru] \ar@{-->}[rru] &  & F(E_2) \ar@/_2pc/[ruu] \ar@{{}*{\cdot\!}>}[u] & \\
    F(C) \ar[r] \ar[u] & F(B) \ar[uu] \ar[ru] \ar@{-->}[ruu] & &
}
\]
Here the solid arrows are the ones we started with (including amalgamation of the squares in $\cd_\cm$). Then, applying Lemma \ref{left-multiadjoint-on-cocones} to the diagram $A \leftarrow C \to B$, we find $D'$ and the dashed arrows. Applying the ``furthermore'' part of Lemma \ref{left-multiadjoint-on-cocones} yields the dotted arrows.

Lemma \ref{left-multiadjoint-on-cocones} also tells us that all the morphisms between objects in the image of $F$ come from morphisms in $\cc$, and such that the diagram in $\cc$ commutes. Furthermore, $\cm$ is left-cancellable, so since the solid arrows are all in $\cm$ we have that the dotted arrows are in $\cm$. We thus see that $D'$ amalgamates the squares in $\cc_{F^{-1}(\cm)}$, as required.
\end{proof}
\begin{theo}\label{left-multiadjoint-lifting}
Let $F: \cc \to \cd$ be a left multiadjoint and let $\cm$ be a left-cancellable composable class of morphisms in $\cd$. If $\nf$ is an independence relation on $\cd_{\cm}$ that satisfies uniqueness then the independence relation $F^{-1}(\nf)$ on $\cc_{F^{-1}(\cm)}$ satisfies uniqueness.
\end{theo}
\begin{proof}
By Proposition \ref{left-multiadjoint-reflects-amalgamating-squares} $F$ reflects amalgamation of squares, so the result follows from Proposition \ref{reflecting-amalgamation-of-squares-lifts-uniqueness}.
\end{proof}
\begin{exams}\label{left-adjoint-examples}
We consider examples of functors $U: \cc \to \cd$ that are left adjoint (and thus in particular left multiadjoint) together with a class of morphisms $\cm$ on $\cd$ such that $\cd_{\cm}$ carries an independence relation that satisfies uniqueness. So Theorem \ref{left-multiadjoint-lifting} applies.
\begin{enumerate}
\item We consider the forgetful functor $U: \Gra \to \Set$, where $\Gra$ is the category of graphs and graph homomorphisms. This functor has a right adjoint given by taking complete graphs. We let $\cm$ be the class of monomorphisms in $\Set$, so $U^{-1}(\cm)$ is the class of monomorphisms in $\Gra$.

Pullback squares form an independence relation on $\Set_{\Mono}$ that satisfies uniqueness. As $U$ preserves and reflects pullback squares, we have that pullback squares also satisfy uniqueness as an independence relation on $\Gra_{\Mono}$. Note that this is not the case of $\Gra_{\Emb}$, where there morphisms are graph embeddings. There uniqueness is known to fail for pullback squares.

\item Let $\cd$ be $\Set$ (resp.\ $\Ab$, the category of abelian groups) and let $\cc$ be the category $\sigma\text{-}\Set$ of sets with an endomorphism (resp.\ $\sigma\text{-}\Ab$, the category of abelian groups with an endomorphism). Then the forgetful functor $U: \cc \to \cd$ preserves limits and colimits, and both $\cc$ and $\cd$ are locally presentable. Thus $U$ is a left adjoint, by \cite[Theorem 1.58]{AR} and the Special Adjoint Functor Theorem. Pullback squares form an independence relation on $\cd_{\Mono}$ that satisfies uniqueness. As $U$ preserves and reflects pullback squares, we have that pullback squares also satisfy uniqueness in $\cc_{\Mono}$.

\item Let $\Coalg_R$ be the category of coalgebras over a ring $R$. The forgetful functor $U:\Coalg_R\to\Mod_R$ to modules over $R$ has a right adjoint. Let $\cm$ consist of morphisms sent by $U$ to monomorphisms ($\cm$ is properly contained in $\Mono$). Since $\Mod_R$ has effective unions, $(\Mod_R)_{\Mono}$ has a stable independence relation given by pullback squares \cite[Theorem 3.1]{LRV1}. In particular the pullback squares satisfy uniqueness, and so the squares in $(\Coalg_R)_{\cm}$ that are sent to pullback squares by $U$ satisfy uniqueness (in fact, they form a stable independence relation, see Theorem \ref{left-multiadjoint-lifts-everything}).
\end{enumerate}
\end{exams}
\begin{exams}
\label{left-multiadjoint-examples}
We give examples of left multiadjoints that are not left adjoints and to which Theorem \ref{left-multiadjoint-lifting} applies.
\begin{enumerate}
\item Fix some $0 < n < \omega$ and a field $K$. Let $\cc$ be the category of vector spaces over $K$ with a predicate for a linear subspace of dimension $n$, where the morphisms are embeddings. That is, its objects are pairs $(U, V)$ with $U$ a linear subspace of the vector space $V$ such that $\dim(V/U) = n$. A morphism $f: (U, V) \to (U', V')$ is a linear injection $f: V \to V'$ such that $f(v) \in U'$ if and only if $v \in U$. Write $\cd$ for the category $\Vec_K \times \Vec_K$ with the morphisms restricted to monomorphisms. We consider the functor $F: \cc \to \cd$, given by $F(U, V) = (U, V/U)$ with the obvious action on morphisms. This is a left multiadjoint: given $(U, W)$ in $\cd$ we let $\{W_i\}_{i \in I}$ be the family of $n$-dimensional linear subspaces of $W$. For each $i \in I$ we consider the object $(U, U \times W_i)$ of $\cc$ and let $e_i: F(U, U \times W_i) \to (U, W)$ be the inclusion $F(U, U \times W_i) \cong (U, W_i) \hookrightarrow (U, W)$. Then the family $\{e_i: F(U, U \times W_i) \to (U, W)\}_{i \in I}$ witnesses that $F$ is a left multiadjoint. The independence relation on $\cd$ given by $\ind^{\lin}$ on each component then lifts along $F$ to an independence relation on $\cc$ that satisfies uniqueness.

To see that $F$ is not a left adjoint we note that it is full and faithful. So if it were a left adjoint, we could view $\cc$ as a coreflective subcategory of $\cd$. So $\cc$ would be closed under coproducts in $\cd$, but for example the coproduct of $(0, K^n)$ and $(0, K^n)$ does not exist in $\cc$.

\item Let $\ca$ be a small category. We say that a presheaf $K: \ca^{\op} \to \Set$ is \emph{connected} if for any $a \in K(A)$ and $a' \in K(A')$ there are $a = a_0, a_1, \ldots, a_n = a'$ in the image of $K$ such that for all $0 \leq i < n$ there is a morphism $f$ in $\ca$ with $K(f)(a_i) = a_{i+1}$ or $K(f)(a_{i+1}) = a_i$. Write $\cc$ for the full subcategory of $\Set^{\ca^{\op}}$ of connected presheaves. The the inclusion functor $F: \cc \hookrightarrow \Set^{\ca^{\op}}$ is a left multiadjoint. Indeed, given a presheaf $K$, the inclusions of connected components of $K$ (i.e., maximal connected subpresheaves) witness that $F$ is a left multiadjoint. The pullback squares form an independence relation on $\Set^{\ca^{\op}}_{\Mono}$ that satisfies uniqueness \cite[Example 5.3]{LRV}. So the pullback squares also form an independence relation on $\cc_{\Mono}$ that satisfies uniqueness. Finally, we note that $F$ cannot be a left adjoint, because then $\cc$ would be a coreflective subcategory of $\Set^{\ca^{\op}}$ and thus has to be closed under colimits in $\Set^{\ca^{\op}}$. However, $\cc$ is clearly not closed under coproducts in $\Set^{\ca^{\op}}$.

\item We finish with a general construction that yields a left multiadjoint that is rarely a left adjoint. Let $\{F_j: \cc_j \to \cd\}_{j \in J}$ be a family of left (multi)adjoints. Then the induced functor $F: \coprod_{j \in J} \cc_j \to \cd$ is a left multiadjoint. To see this, we fix an object $D$ in $\cd$. For each $j \in J$ there is a family $\{e^j_i: F(C^j_i) \to D\}_{i \in I_j}$ witnessing that $F_j$ is a left multiadjoint. We claim that $\{e^j_i\}_{j \in J, i \in I_j}$ witnesses that $F$ is a left multiadjoint. Let $e: F(C) \to D$ be some morphism. Then $C$ is an object in $\cc_j$ for some $j \in J$. Hence $e$ factors as $e^j_i F(f)$ for some $i \in I_j$ and $f: C \to C^j_i$. Clearly, there are no solutions for this factorisation problem in $I_{j'}$ or $\cc_{j'}$, for $j' \neq j$. So $i$ and $f$ remain the unique solution.

An $F$ as above is rarely a left adjoint. For example, if $\cd$ has the \emph{joint continuation property}: any two objects $A$ and $B$ admit a cospan into a third object $A \to C \leftarrow B$. Let $j,j' \in J$ be distinct (of course, if $|J| = 1$ we are in a trivial case) and pick objects $C_j$ and $C_{j'}$ in $\cc_j$ and $\cc_{j'}$ respectively (again, we exclude the trivial case where one of the categories is empty). Let $D$ be such that there are morphisms $F(C_j) \to D \leftarrow F(C_{j'})$. If $F$ were a left adjoint, with right adjoint $R: \cd \to \coprod_{j \in J} \cc_j$, then we would get a diagram $C_j \to R(D) \leftarrow C_{j'}$ in $\coprod_{j \in J} \cc_j$, which is impossible because $C_j$ and $C_{j'}$ live in different connected components.

To give a concrete example of this construction we consider $\cc = \coprod_{n < \omega} \Set^{\mathbb{Z}^n}$, the category of sets with a finite number of automorphisms. By the same argument as in Example \ref{left-adjoint-examples}(2) the forgetful functor $\Set^{\mathbb{Z}^n} \to \Set$ is left adjoint for each $n < \omega$. By the above the forgetful functor $U: \cc \to \Set$ is a left multiadjoint that is not left adjoint. The independence relation on $\Set_{\Mono}$ given by pullback squares then lifts to an independence relation (again given by pullback squares) on $\cc$ that satisfies uniqueness.
\end{enumerate}
\end{exams}

\section{Lifting existence}\label{sec:lifting-existence}
Left multiadjoints will generally also lift the existence property, in a similar way to how they lift the uniqueness property. As in that setting, we will want to consider a category $\cc$ and some class of morphisms $\cm$, where $\cc_\cm$ actually carries the independence relation $\nf$ (see also the discussion before Proposition \ref{left-multiadjoint-reflects-amalgamating-squares}). We can also view $\nf$ as an independence relation on $\cc$ by just taking the same collection of squares, but it will lose certain properties, such as invariance. In practice we often retain a weaker version of invariance.
\begin{defi}\label{semi-invariance}
An independence relation $\nf$ on some category $\cc$ is said to satisfy \emph{semi-invariance} if given any commuting diagram like below, the outer square being $\nf$-independent implies that the inner square is $\nf$-independent.
\[
\xymatrix@=1pc{
A \ar[r] \ar@/^1pc/[rr] & D \ar[r] & E \\
C \ar[r] \ar[u] & B \ar[u] \ar@/_1pc/[ur] &
}
\]
\end{defi}
The philosophy behind the above definition is that if something was already dependent in $D$, such as some equation holding, then its image in $E$ is also dependent, as morphisms generally preserve truth of things like equations. The contrapositive of this statement is exactly what we called semi-invariance. Invariance is then the statement that independence is also preserved, something that goes wrong more quickly, as the following example illustrates.
\begin{exam}\label{semi-invariance-example}
Let $\nf$ on $\Set_{\Mono}$ be given by pullback squares. This is easily seen to have invariance. However, $\nf$ as an independence relation on $\Set$ consists of all those pullback squares whose morphisms are monomorphisms. As such, $\nf$ will no longer satisfies invariance as is illustrated in the diagram below: the inner square is a pullback, but the outer square is not, and so independence is not preserved upwards along the dashed arrow.
\[
\xymatrix@=2pc{
\{a\} \ar@{^{(}->}[r] \ar@/^1.5pc/[rr] & \{a,b\} \ar@{-->}[r] & \{ * \} \\
\; \emptyset \; \ar@{^{(}->}[r] \ar@{^{(}->}[u] & \{b\} \ar@{^{(}->}[u] \ar@/_1pc/[ur] &
}
\]
However, $\nf$ on $\Set$ does satisfy semi-invariance, as is straightforwardly verified.
\end{exam}
\begin{theo}\label{left-multiadjoint-lifts-existence}
Let $F: \cc \to \cd$ be a left multiadjoint and let $\cm$ be a composable class of morphisms in $\cd$. Suppose that $\nf$ is an independence relation on $\cd_\cm$ that satisfies existence and that it satisfies semi-invariance as an independence relation on $\cd$. Then $F^{-1}(\nf)$ as an independence relation on $\cc_{F^{-1}(\cm)}$ satisfies existence.
\end{theo}
\begin{proof}
Let $A \leftarrow C \to B$ be any span in $\cc_{F^{-1}(\cm)}$. We then find the commuting diagram below.
\[
\xymatrix@=2pc{
F(A) \ar@{{}*{\cdot}>}[r] \ar@/^1pc/@{-->}[rr] & F(D) \ar@{{}*{\cdot}>}[r] & E \\
F(C) \ar[r] \ar[u] & F(B) \ar@{{}*{\cdot\!}>}[u] \ar@/_1pc/@{-->}[ur] &
}
\]
By existence for $\nf$ we find $E$ and the dashed arrows, such that the outer square is $\nf$-independent. We then apply Lemma \ref{left-multiadjoint-on-cocones} to the span $A \leftarrow C \to B$ to find $D$ and the dotted arrows, making everything commute.

Furthermore, Lemma \ref{left-multiadjoint-on-cocones} also implies that the morphisms between objects in the image of $F$ come from morphisms in $\cc$, and such that the diagram in $\cc$ commutes. By semi-invariance the inner square is $\nf$-independent. In particular, that means that it lies in the image of the inclusion $\cd_\cm \hookrightarrow \cd$. Therefore, the square with $A, B, C, D$ in $\cc$ also lies in $\cc_{F^{-1}(\cm)}$ and is $F^{-1}(\nf)$-independent, as required.
\end{proof}

We give another sufficient criterion to lift existence, which is motivated by practical examples.
\begin{defi}\label{admitting-1-completions}
Let $F: \cc \to \cd$ be a faithful functor. Let $f: F(A) \to B$ be a morphism in $\cd$. A \emph{$1$-completion} is a morphism $g: A \to C$ in $\cc$ such that $F(g): F(A) \to F(C)$ factors through $f$ (i.e., $F(g) = hf$ for some $h: B \to F(C)$). We say that $F$ \emph{admits $1$-completions} if for any $F(A) \to B$ in $\cd$ there is a $1$-completion.
\[
\xymatrix@=3pc{
F(A) \ar[r]^f \ar@{-->}@/^2pc/[rr]^{F(g)} & B \ar@{-->}[r]^h & F(C)
}
\]
\end{defi}
For the intuition behind $1$-completions we consider the forgetful functor $F: \Bil_K \to \Vec_K$ from Example \ref{motivating-examples}(1). Suppose that we have a vector space $V$ and a subspace $U \subseteq V$ with a bilinear form $[-, -]$ defined on $U$. Then we can extend $[-, -]$ to all of $V$, making $V$ into a bilinear space $V'$. In terms of the functor $F$ this corresponds to having a morphism $F(U) \to V$ and then we find a morphism $V \to F(V')$ (in fact, the identity in this case) such that $F(U \to V')$ is the composite $F(U) \to V \to F(V')$.

More generally, admitting $1$-completions says the following. Suppose that we are given an object $B$ from $\cd$ on which the extra structure that objects in $\cc$ carry is partially defined, namely as some $F(A) \to B$. Then $B$ can be completed to an object $F(C)$ from $\cc$ in a way that is compatible with the already defined structure.

Now consider the following higher dimensional situation for $F: \Bil_K \to \Vec_K$. Suppose we have a vector space $V$ with a subspaces $U, U_1, U_2 \subseteq V$ such that $U \subseteq U_1 \cap U_2$ and each carries a bilinear form, with the bilinear forms on $U_1$ and $U_2$ extending the one on $U$. To extend all these bilinear forms to one bilinear form on $V$ we generally need $U_1$ to be linearly independent from $U_2$ over $U$, because that gives us complete freedom in how we define the bilinear form on pairs of vectors $(u_1, u_2)$ and $(u_2, u_1)$ with $u_1 \in U_1 \setminus U$ and $u_2 \in U_2 \setminus U$. Diagrammatically, this means that given an independent square in $\Vec_K$ like below on the left there is some commuting square in $\Bil_K$ like below on the right whose image under $F$ is the square on the left.
\[
\vcenter{\vbox{\xymatrix@=2pc{
  F(U_1) \ar[r] & V \\
  F(U) \ar[u] \ar[r] \ar@{}[ur]|-{\nf^{\lin}} & F(U_2) \ar[u]
}}}
\quad \rightsquigarrow \quad
\vcenter{\vbox{\xymatrix@=2pc{
  U_1 \ar[r] & V' \\
  U \ar[u] \ar[r] & U_2 \ar[u]
}}}
\]
\begin{defi}\label{admitting-2-completions}
Let $F: \cc \to \cd$ be a faithful functor. Suppose we are given a commuting square in $\cd$ like the solid arrows in the diagram below, where the morphisms between objects in the image of $F$ are also in the image of $F$. A \emph{$2$-completion} for that square consists of an object $E$ and the dashed arrows making the diagram below commute.
\[
\xymatrix@=2pc{
  F(A) \ar[r] \ar@{-->}@/^1.5pc/[rr]^{F(f)} & D \ar@{-->}[r]_(.4)h & F(E) \\
  F(C) \ar[u] \ar[r] \ar@{}[ur]|-{\nf} & F(B) \ar[u] \ar@{-->}@/_1pc/[ur]_{F(g)} &
}
\]
Given an independence relation $\nf$ on $\cd$ we say that $F$ \emph{admits $2$-completions with respect to $\nf$} if for any $\nf$-independent square in $\cd$ there is a $2$-completion. There will often only be one relevant independence relation $\nf$, in which case we will just say that $F$ admits $2$-completions.
\end{defi}
We note that faithfulness of $F$ implies that commutativity of the diagrams in the above definition means that the corresponding square in $\cc$ (i.e., the one involving $C$, $A$, $B$ and $E$) also commutes.

Observe that our notion of admitting $2$-completions is weaker than the notion of cofinality given in \cite[Definition 2.4]{LRV2}, even if we would fix the shape of the diagram there. This is because we require the diagram to also be independent, the necessity of which we discussed before Definition \ref{admitting-2-completions}. That being said, the different notions have different intended applications, as is indicated by the different names.
\begin{propo}\label{admitting-2-completions-implies-admitting-1-completions}
Let $F: \cc \to \cd$ be a faithful functor and suppose that $\nf$ is an independence relation on $\cd$ satisfying basic existence. If $F$ admits $2$-completions it admits $1$-completions.
\end{propo}
\begin{proof}
This is done by inserting identity morphisms in the right places, which is why we need basic existence to have that these squares be independent.
\end{proof}
\begin{theo}\label{admitting-2-completions-lifts-existence}
Let $F: \cc \to \cd$ be a faithful functor and suppose that $\nf$ is an independence relation on $\cd$ that satisfies invariance.
\begin{enumerate}
\item If $F$ admits $2$-completions and $\nf$ satisfies existence then $F^{-1}(\nf)$ satisfies existence.
\item If $F$ admits $1$-completions, $\nf$ satisfies uniqueness and $F^{-1}(\nf)$ satisfies existence then $F$ admits $2$-completions.
\end{enumerate}
\end{theo}
\begin{proof}
We first prove (1). Let $A \leftarrow C \to B$ be a span in $\cc$. By existence and admitting $2$-completions we can recreate the diagram from Definition \ref{admitting-2-completions}. By invariance the commuting square below on the left is independent. As $F$ is faithful, the square on the right also commutes and is thus an independent square, as required.
\[
\vcenter{\vbox{\xymatrix@=2pc{
  F(A) \ar[r] & F(E) \\
  F(C) \ar[u] \ar@{}[ur]|-{\nf} \ar[r] & F(B) \ar[u]
}}}
\quad \quad
\vcenter{\vbox{\xymatrix@=2.6pc{
  A \ar[r] & E \\
  C \ar[u] \ar@{}[ur]|-{F^{-1}(\nf)} \ar[r] & B \ar[u]
}}}
\]
For (2) we suppose that we have an independent diagram like on the left below. By existence for $F^{-1}(\nf)$ we get an independent diagram like on the right below.
\[
\vcenter{\vbox{\xymatrix@=2pc{
  F(A) \ar[r] & D \\
  F(C) \ar[u] \ar@{}[ur]|-{\nf} \ar[r] & F(B) \ar[u]
}}}
\quad \quad
\vcenter{\vbox{\xymatrix@=2.6pc{
  A \ar[r] & D' \\
  C \ar[u] \ar@{}[ur]|-{F^{-1}(\nf)} \ar[r] & B \ar[u]
}}}
\]
That means that the image of the diagram on the right under $F$ is $\nf$-independent. So we can apply uniqueness for $\nf$ to the diagram of solid arrows below to get the dashed arrows making everything commute. The dotted arrow is then found by taking a $1$-completion of $F(D') \to E$.
\[
\xymatrix@=2pc{
  & D \ar@{-->}[r] & E \ar@{{}*{\cdot}>}[r] & F(E') \\
  F(A) \ar[rr] \ar[ur] & & F(D') \ar@{-->}[u] & \\
  F(C) \ar[u] \ar[r] & F(B) \ar[ur] \ar[uu] & &
}
\]
This yields a $2$-completion of the original diagram by taking $D \to E \to F(E')$, $F(A) \to F(D') \to E \to F(E')$ and $F(B) \to F(D') \to E \to F(E')$. These last two morphisms are in the image of $F$ because $F(D') \to E \to F(E')$ is in the image of $F$ by our choice of $1$-completion.
\end{proof}

\section{Strong 3-amalgamation}
\label{sec:strong-3-amalgamation}
\begin{defi}\label{strong-3-amalgamation}
We say that an independence relation $\nf$ on a category $\cc$ satisfies \emph{strong $3$-amalgamation} if given an $\nf$-independent horn like the solid arrows below, there are $N$ and the dashed arrows such that everything commutes and every square in the resulting cube is $\nf$-independent (we call this an \emph{$\nf$-independent cube}).
\[
\xymatrix@=2pc{
        & N_2 \ar @{-->}[rr] & & N\\
        A \ar[rr] \ar[ur] &
        & N_1 \ar @{-->}[ur] & \\
        & C \ar'[u][uu] \ar'[r][rr] & & N_3 \ar@{-->} [uu] & \\
        M \ar[ur]\ar[rr]\ar[uu] & & B \ar[ur]\ar[uu] &  
      }
\]
\end{defi}
The name \emph{strong} 3-amalgamation is because in the definition of $3$-amalgamation it is only required that $A \nf^N_M N_3$. If $\nf$ satisfies transitivity (which the independence relations we consider generally will) then the independence of this diagonal square automatically follows. Conversely, it turns out that under reasonable assumptions $3$-amalgamation implies strong $3$-amalgamation, which we will prove in the remainder of this section.
\begin{theo}\label{strong-3-amal-from-3-amal}
Let $\cc$ be a category whose morphisms are monomorphisms that has binary joins of subobjects and suppose that it is equipped with an independence relation $\nf$ that satisfies invariance, monotonicity, transitivity, symmetry, existence, base monotonicity and $3$-amalgamation. Then $\nf$ satisfies strong $3$-amalgamation.
\end{theo}
\begin{lemma}\label{3-amal-replace-arrow}
Let $\cc$ and $\ind$ be as in Theorem \ref{strong-3-amal-from-3-amal}. Then given a commuting cube like below on the left with $A \nf^N_M N_3$ there are $N \to N'$ and $f: N_2 \to N'$ such that the square below on the right is independent
\[
\vcenter{\vbox{\xymatrix@=2pc{
        & N_2 \ar[rr] \ar@{-->}@/^1pc/[rrr]^f & & N \ar@{-->}[r] & N'\\
        A \ar[rr] \ar[ur] &
        & N_1 \ar[ur] & \\
        & C \ar'[u][uu] \ar'[r][rr] & & N_3 \ar[uu] & \\
        M \ar[ur]\ar[rr]\ar[uu] & & B \ar[ur]\ar[uu] &  
}}}
\quad \quad
\vcenter{\vbox{\xymatrix@=3pc{
  N_2 \ar[r]^f & N' \\
  C \ar[u] \ar@{}[ur]|-{\nf} \ar[r] & N_3 \ar[u]
}}}
\]
and such that $A \to N_2 \to N \to N'$ is the same as $A \to N_2 \xrightarrow{f} N'$, and similarly with $C$ in place of $A$.
\end{lemma}
Note that the above lemma does not claim that the diagram with the dashed arrows commutes. Also, the proof of the above lemma does not use $3$-amalgamation or joins of subobjects.
\begin{proof}
We build the up the commuting diagram below in a few steps.
\[
\xymatrix@=2pc{
& N^* \ar@{~>}[r]^g & N' \\
N_2 \ar@{{}*{\cdot}>}[ur] & P \ar@{-->}[r] \ar@{{}*{\cdot\!}>}[u] \ar@{}[ur]|-{\nf} & N^* \ar@{~>}[u]_h \\
A \ar@/_0.75pc/[rr] \ar[u] \ar@{-->}[ur] & & N \ar@{-->}[u] \\
M \ar[r] \ar[u] & C \ar[r] \ar@{-->}[uu] \ar@{}[uur]|-{\nf} \ar[uul] & N_3 \ar[u]
}
\]
\begin{enumerate}
\item We start with the solid arrows.
\item Applying base monotonicity to $A \nf^N_M N_3$ we find the dashed arrows, such that $P \nf^{N^*}_C N_3$.
\item By commutativity of the original cube and what we have constructed so far, we have that $A \to N_2 \to N \to N^*$ is the same as $A \to P \to N^*$, and similarly with $C$ in place of $A$. This yields the dotted arrows: $\xymatrix{P \ar@{{}*{\cdot}>}[r] & N^*}$ is the same as $P \dashrightarrow N^*$ and $\xymatrix{N_2 \ar@{{}*{\cdot}>}[r] & N^*}$ is the composition $N_2 \to N \dashrightarrow N^*$.
\item Applying existence to $\xymatrix{N^* & P \ar@{{}*{\cdot}>}[l] \ar@{-->}[r] & N^*}$ we find the squigly arrows making the top square independent.
\end{enumerate}
We take the unnamed morphism $N \to N'$ to be the composition $N \to N^* \xrightarrow{h} N'$ and we take $f: N_2 \to N'$ to be the composition $N_2 \to N^* \xrightarrow{g} N'$. By transitivity (and symmetry) the right rectangle in the above diagram is independent, so by monotonicity we have that the required square with $f$ is indeed independent. Straightforwardly writing out definitions then verifies the claim about the morphisms with domains $A$ and $C$.
\end{proof}

\begin{lemma}\label{base-monotonicity-with-joins}
Let $\cc$ be a category whose morphisms are monomorphisms that has binary joins of subobjects. Suppose that $\nf$ is an independence relation on $\cc$ satisfying invariance, monotonicity and base monotonicity. Then given any independent square like below on the left the square below on the right is independent.
\[
\vcenter{\vbox{\xymatrix@=2pc{
A \ar[rr] & & M \\
C \ar[u] \ar[r] \ar@{}[urr]|-{\nf} & B' \ar[r] & B \ar[u]
}}}
\quad \implies \quad
\vcenter{\vbox{\xymatrix@=3pc{
  A \vee B' \ar[r] & M \\
  B' \ar[r] \ar[u] \ar@{}[ur]|-{\nf} & B \ar[u]
}}}
\]
\end{lemma}
\begin{proof}
By base monotonicity we find the dashed arrows below making everything commute and such that $P \nf^N_{B'} B$.
\[
\xymatrix@=2pc{
& P \ar@{-->}[r] & N \\
A \ar@/_0.75pc/[rr] \ar@{-->}[ur] & & M \ar@{-->}[u] \\
C \ar[u] \ar[r] & B' \ar[r] \ar@{-->}[uu] \ar@{}[uur]|-{\nf} & B \ar[u]
}
\]
By monotonicity then $A \vee B' \nf^N_{B'} B$ and the independence of the required square follows by invariance and the fact that it does not matter whether we compute $A \vee B'$ in $M$ or $N$.
\end{proof}
\begin{proof}[Proof of Theorem \ref{strong-3-amal-from-3-amal}]
Suppose we are given an independent horn like the solid arrows below. By $3$-amalgamation we then find the dashed arrows making the cube commute and such that $A \nf^N_M N_3$.
\[
\xymatrix@=2pc{
        & N_2 \ar @{-->}[rr] & & N\\
        A \ar[rr] \ar[ur] &
        & N_1 \ar @{-->}[ur] & \\
        & C \ar'[u][uu] \ar'[r][rr] & & N_3 \ar@{-->} [uu] & \\
        M \ar[ur]\ar[rr]\ar[uu] & & B \ar[ur]\ar[uu] &  
      }
\]
Applying Lemma \ref{3-amal-replace-arrow} twice we may assume that that $N_1 \nf^N_B N_3$ and $N_2 \nf^N_C N_3$. Note that $A \vee B \to N$ factors through $N_1$, so by existence we find the following independent diagram.
\begin{align}
\label{strong-3-amal-from-3-amal:existence-on-a-join-b}
\vcenter{\vbox{\xymatrix@=2pc{
N \ar@{-->}[r] & N' \\
A \vee B \ar[u] \ar[r] \ar@{}[ur]|-{\nf} & N_1 \ar@{-->}[u]_f
}}}
\end{align}
We claim that the cube below is an independent cube, where $N_2 \to N'$ is given by composition of the unnamed morphisms $N_2 \to N \to N'$, and similarly for $N_3$.
\[
\xymatrix@=2pc{
        & N_2 \ar[rr] & & N'\\
        A \ar[rr] \ar[ur] &
        & N_1 \ar[ur]|-{\,f\,} & \\
        & C \ar'[u][uu] \ar'[r][rr] & & N_3 \ar [uu] & \\
        M \ar[ur]\ar[rr]\ar[uu] & & B \ar[ur]\ar[uu] &  
      }
\]
Firstly, by commutativity of (\ref{strong-3-amal-from-3-amal:existence-on-a-join-b}) we have that $A \to N_1 \to N \to N'$ is the same as $A \to N_1 \xrightarrow{f} N'$, and similarly for $B$. This establishes commutativity of the cube. We are thus left to verify the independence of the two squares that include $f$.
\begin{itemize}
\item By transitivity we have $B \nf^N_M N_2$. Applying Lemma \ref{base-monotonicity-with-joins} to this and the factorisation $M \to A \to N_2$ we find $N_2 \nf^N_A A \vee B$ (after an additional application of symmetry). We thus obtain two composable independent squares like below, where the right square is the one from (\ref{strong-3-amal-from-3-amal:existence-on-a-join-b}).
\[
\xymatrix@=1.5pc{
N_2 \ar[r] & N \ar[r] & N' \\
A \ar[r] \ar[u] \ar@{}[ur]|-{\nf} & A \vee B \ar[r] \ar[u] \ar@{}[ur]|-{\nf} & N_1 \ar[u]_f
}
\]
By transitivity the outer rectangle is indeed independent.
\item From $N_1 \nf^N_B N_3$ in the original cube we have by monotonicity that $A \vee B \nf^N_B N_3$. Applying symmetry to this we can again compose with the independent square from (\ref{strong-3-amal-from-3-amal:existence-on-a-join-b}).
\[
\xymatrix@=1.5pc{
N_3 \ar[r] & N \ar[r] & N' \\
B \ar[r] \ar[u] \ar@{}[ur]|-{\nf} & A \vee B \ar[r] \ar[u] \ar@{}[ur]|-{\nf} & N_1 \ar[u]_f
}
\]
So again by transitivity the outer rectangle is indeed independent. \qedhere
\end{itemize}
\end{proof}
Our proof relies on binary joins of subobjects, but there are natural examples where these do not exist. This already happens for categories of the form $\Mod(T)$, where $T$ is a first-order theory. For example, if $T$ is the theory of the random graph: the theory of an infinite graph where for any two finite disjoint subsets $A$ and $B$ there is a vertex that has an edge to every vertex in $A$ and to none in $B$. Let $M$ be a sufficiently big model ($\omega$-saturated will do), then we can find two copies of the countable random graph $M_0, M_1 \subseteq M$ such that there are no edges between $M_0$ and $M_1$. Clearly, the join $M_0 \vee M_1$ does not exist in $\Mod(T)$. At the same time, the pullback squares form a simple independence relation on $\Mod(T)$ with strong $3$-amalgamation. This motivates the following question. We also note that $\Gra_{\Emb}$ carries the same model-theoretic content, but is much nicer as a category (e.g., it does have binary joins of subobjects).
\begin{question}
Is Theorem \ref{strong-3-amal-from-3-amal} still true when we remove the assumption about having binary joins of subobjects?
\end{question}
Finally, we comment on the assumption that all morphisms in $\cc$ are monomorphisms. This is to use joins of subobjects as in Lemma \ref{base-monotonicity-with-joins}. This is enough for our intended applications. Furthermore, most categories that one would study model-theoretically using an independence relation have only monomorphisms. However, the above proofs could be carried out in any category, when replacing joins of subobjects by some weakened version of having pushouts.

\section{Lifting 3-amalgamation}
\label{sec:lifting-3-amalgamation}
\begin{theo}\label{left-multiadjoint-lifts-3-amalgamation}
Let $F: \cc \to \cd$ be a left multiadjoint and let $\cm$ be a composable class of morphisms in $\cd$. Suppose that $\nf$ is an independence relation on $\cd_\cm$ that satisfies (strong) $3$-amalgamation and that it satisfies semi-invariance as an independence relation on $\cd$. Then $F^{-1}(\nf)$ as an independence relation on $\cc_{F^{-1}(\cm)}$ satisfies (strong) $3$-amalgamation.
\end{theo}
\begin{proof}
Analogous to Theorem \ref{left-multiadjoint-lifts-existence}.
\end{proof}
In Theorem \ref{admitting-2-completions-lifts-existence} we saw that ``admitting $2$-completions'' is enough to lift the existence property. Viewing the existence property as $2$-amalgamation, we can push this to higher dimensions. We thus formulate a $3$-dimensional version below, with the same intuition (see the discussion before Definition \ref{admitting-2-completions}) and one can again use the same example of bilinear spaces to see what happens in practice. We will have no need for any higher dimensional versions.
\begin{defi}\label{admitting-3-completions}
Let $F: \cc \to \cd$ be a faithful functor. Suppose we are given a commuting cube in $\cd$ like the solid arrows in the diagram below, where the morphisms between objects in the image of $F$ are also in the image of $F$. A \emph{$3$-completion} for that cube consists of an object $N^*$ and the dashed arrows making the diagram below commute.
\[
\xymatrix@=2pc{
  & F(N_1) \ar[rr] \ar@{-->}@/^2pc/[rrr]^{F(n_1)} & & N \ar@{-->}[r]_(0.4)h & F(N^*) \\
  F(A) \ar[ur] \ar[rr] & & F(N_2) \ar[ur] \ar@{-->}@/_/[urr]_(.6){F(n_2)} & & \\
  & F(B) \ar[uu] \ar[rr] & & F(N_3) \ar[uu] \ar@{-->}@/_1pc/[uur]_{F(n_3)} & \\
  F(M) \ar[uu] \ar[rr] \ar[ur] & & F(C) \ar[uu] \ar[ur] & &
}
\]
Given an independence relation $\nf$ on $\cd$ we say that $F$ \emph{admits $3$-completions with respect to $\nf$} if for any $\nf$-independent cube in $\cd$ there is a $3$-completion. There will often only be one relevant independence relation $\nf$, in which case we will just say that $F$ admits $3$-completions.
\end{defi}
\begin{propo}\label{admitting-3-completions-implies-admitting-2-completions}
Let $F: \cc \to \cd$ be a faithful functor and suppose that $\nf$ is an independence relation on $\cd$ satisfying basic existence. If $F$ admits $3$-completions then it admits $2$-completions (and hence $1$-completions).
\end{propo}
\begin{proof}
Analogous to Proposition \ref{admitting-2-completions-implies-admitting-1-completions}.
\end{proof}
\begin{theo}\label{admitting-3-completions-lifts-strong-3-amalgamation}
Let $F: \cc \to \cd$ be a faithful functor. Suppose that $\nf$ is an independence relation on $\cd$ that satisfies invariance and strong $3$-amalgamation. If $F$ admits $3$-completions then $F^{-1}(\nf)$ satisfies strong $3$-amalgamation.
\end{theo}
\begin{proof}
Analogous to Theorem \ref{admitting-2-completions-lifts-existence}(1).
\end{proof}
Unlike Theorem \ref{admitting-2-completions-lifts-existence}(2), we have not formulated a converse for Theorem \ref{admitting-3-completions-lifts-strong-3-amalgamation}. To replicate the argument for Theorem \ref{admitting-2-completions-lifts-existence}(2) we would require some form of $3$-uniqueness, which is much more subtle.
\begin{defi}\label{independent-galois-type-amalgamation}
Let $F: \cc \to \cd$ be a faithful functor and $\nf$ an independence relation on $\cd$. We say that $F$ has \emph{$\nf$-independent horn amalgamation} if the following holds. Suppose we are given an $\nf$-independent horn in the image of $F$ and an $\nf$-independent square like below.
\[
\vcenter{\vbox{\xymatrix@=1pc{
  & F(N_1) & & \\
  F(A) \ar[ur]^{F(a_1)} \ar[rr]^>>>>>{F(a_2)} & & F(N_2) & \\
  & F(B) \ar[uu]^<<<<{F(b_1)} \ar[rr]_<<<<{F(b_3)} & & F(N_3) \\
  F(M) \ar[uu] \ar[rr] \ar[ur] & & F(C) \ar[uu]_>>>{F(c_2)} \ar[ur]_{F(c_3)} &
}}}
\text{ and }
\vcenter{\vbox{\xymatrix@=2pc{
  F(A) \ar[r] & N \\
  F(M) \ar[u] \ar[r] & F(N_3) \ar[u]
}}}
\]
Then there exists a $2$-completion $N \to F(N^*)$ with $a^*: A \to N^*$ and $n_3: N_3 \to N^*$ for the square on the right, such that the squares below can be amalgamated in $\cc$. That is, the dashed arrows below exist and make everything commute.
\[
\vcenter{\vbox{\xymatrix@=2pc{
  & N_1 \ar@{-->}[r] & \bullet \\
  A \ar[rr]_<<<{a^*} \ar[ur]^{a_1} & & N^* \ar@{-->}[u] \\
  M \ar[u] \ar[r] & B \ar[ur]_{n_3 b_3} \ar[uu]^<<<{b_1} &
}}}
\quad \text{ and } \quad
\vcenter{\vbox{\xymatrix@=2pc{
  & N_2 \ar@{-->}[r] & \bullet \\
  A \ar[rr]_<<<{a^*} \ar[ur]^{a_2} & & N^* \ar@{-->}[u] \\
  M \ar[u] \ar[r] & C \ar[ur]_{n_3 c_3} \ar[uu]^<<<{c_2} &
}}}
\]
\end{defi}

\begin{propo}\label{independent-amalgamation-equivalent}
Let $F: \cc \to \cd$ be a faithful functor with $\cc$ having the amalgamation property. Let $\nf$ an independence relation on $\cd$ satisfying invariance. Then the following condition is equivalent to $F$ having $\nf$-independent horn amalgamation. Suppose we are given a commuting diagram consisting of the solid arrows below, where the morphisms between objects in the image of $F$ are also in the image of $F$, and every square of solid arrows is $\nf$-independent. Then we can find the dashed arrows, making everything commute.
\[
\xymatrix@=1pc{
  & F(N_1) \ar@{-->}@/^1pc/[rrr]^{F(n_1)} & & N \ar@{-->}[r]_(0.4)h & F(N^*) \\
  F(A) \ar[ur] \ar[rr] \ar[urrr] & & F(N_2) \ar@{-->}@/_1pc/[urr]_(0.65){F(n_2)} & & \\
  & F(B) \ar[uu] \ar[rr] & & F(N_3) \ar[uu] \ar@{-->}@/_1pc/[uur]_{F(n_3)} & \\
  F(M) \ar[uu] \ar[rr] \ar[ur] & & F(C) \ar[uu] \ar[ur] & &
}
\]
\end{propo}
\begin{proof}
In the proof below we implicitly use that $F$ is faithful, which implies that a diagram in $\cc$ commutes if and only if its image under $F$ commutes in $\cd$.

To see that this condition implies $\nf$-independent horn amalgamation, we take $a^*$ to be $A \to N_1 \xrightarrow{n_1} N^*$. Then $h$ together with $a^*$ and $n_3$ forms the required $2$-completion.

Conversely, we can apply $\nf$-independent horn amalgamation to obtain a $2$-completion $N'$ for the square involving $N$. Then, working in $\cc$, the dashed arrows in the diagram below exist by assumption. Amalgamating over $N'$ then yields the dotted arrows, which gives the required $N^*$.
\[
\xymatrix@=1pc{
 & & & \bullet \ar@{{}*{\cdot}>}[r] & N^* \\
 & N_1 \ar@/^1pc/@{-->}[urr] & & N' \ar@{-->}[u] \ar@{-->}[r] & \bullet \ar@{{}*{\cdot\!}>}[u] \\
 A \ar[rr] \ar[ur] \ar[urrr] & & N_2 \ar@/_1pc/@{-->}[urr] & & \\
 & B \ar[uu] \ar[rr] & & N_3 \ar[uu] & \\
 M \ar[rr] \ar[uu] \ar[ur] & & C \ar[uu] \ar[ur] & &
}
\]
\end{proof}
The formulation of $\nf$-independent horn amalgamation might look more complicated than admitting 3-completions. However, it is sometimes more practical, as illustrated in the example below.
\begin{exam}\label{graphs-independent-horn-amalgamation}
Consider the category $\Gra_{\Emb}$ of graphs and graph embeddings and the forgetful functor $F: \Gra_{\Emb} \to \Set_{\Mono}$. One easily checks that two squares of graphs, like in the diagram below, can be amalgamated if and only if for any $a \in A$ and $b \in B$ they have an edge in $N$ precisely when they have an edge in $N'$ (compare this also to the discussion about types, Remark \ref{reflecting-amalgamation-of-squares-vs-types}).
\[
\xymatrix@=1pc{
 & N & \\
 A \ar[rr] \ar[ur] & & N' \\
 M \ar[u] \ar[r] & B \ar[uu] \ar[ur] &
}
\]
Let $\nf$ on $\Set_{\Mono}$ be given by those squares that are pullbacks. We will show that $F$ has $\nf$-independent horn amalgamation. Suppose we are given a horn and square like in Definition \ref{independent-galois-type-amalgamation}. We make $N$ into a graph $N^*$ as follows. For any $x, y \in F(A)$ (resp.\ any $x,y \in F(N_3)$) we have an edge between $x$ and $y$ if and only if they have an edge in $A$ (resp.\ in $N_3$). Furthermore, for any $a \in F(A)$ and any $b \in F(B) \subseteq F(N_3)$ (resp.\ any $c \in F(C) \subseteq F(N_3)$) we have an edge if and only if $a$ and $b$ have an edge in $N_1$ (resp.\ $a$ and $c$ have an edge in $N_2$). This last construction is well-defined because $B \cap C = M$ in $N_3$. There are no other edges in $N^*$. This construction makes $N^*$ into a $2$-completion and the two relevant squares can be amalgamated by our previous condition for square amalgamation.
\end{exam}
\begin{theo}\label{independent-amalgamation-and-uniqueness}
Let $F: \cc \to \cd$ be a faithful functor and suppose that $\nf$ is an invariant independence relation on $\cd$. If $\nf$ satisfies $3$-amalgamation and $F$ has $\nf$-independent horn amalgamation then $F^{-1}(\nf)$ satisfies $3$-amalgamation.

In the other direction we have that if $\nf$ satisfies uniqueness, $F$ admits $1$-completions and $F^{-1}(\nf)$ satisfies $3$-amalgamation then $F$ satisfies $\nf$-independent horn amalgamation.
\end{theo}
\begin{proof}
Given an independent horn in $\cc$, its image is an independent horn in $\cd$. Applying $3$-amalgamation of $\nf$ then yields a commuting cube like below on the left, such that the square like below on the right is independent.
\[
\vcenter{\vbox{\xymatrix@=1pc{
  & F(N_1) \ar[rr] & & N \\
  F(A) \ar[ur] \ar[rr] & & F(N_2) \ar[ur] & \\
  & F(B) \ar[uu] \ar[rr] & & F(N_3) \ar[uu] \\
  F(M) \ar[uu] \ar[rr] \ar[ur] & & F(C) \ar[uu] \ar[ur] &
}}}
\quad \quad \quad
\vcenter{\vbox{\xymatrix@=2pc{
  F(A) \ar[r] & N \\
  F(M) \ar[u] \ar[r] & F(N_3) \ar[u]
}}}
\]
Applying $\nf$-independent horn amalgamation---or more precisely: the condition from Proposition \ref{independent-amalgamation-equivalent}---we obtain a commuting cube in $\cc$ like below
\[
\xymatrix@=1pc{
  & N_1 \ar[rr] & & N^* \\
  A \ar[ur] \ar[rr] & & N_2 \ar[ur] & \\
  & B \ar[uu] \ar[rr] & & N_3 \ar[uu] \\
  M \ar[uu] \ar[rr] \ar[ur] & & C \ar[uu] \ar[ur] &
}
\]
By invariance   $F(A) \nf^{F(N^*)}_{F(M)} F(N_3)$. So by construction of the above cube, its diagonal square involving $M$, $A$, $N_3$ and $N^*$ is $F^{-1}(\nf)$-independent, proving $3$-amalgamation.

Conversely, suppose that $\nf$ satisfies uniqueness, $F$ admits $1$-completions and that $F^{-1}(\nf)$ satisfies $3$-amalgamation. Let the setup be as in Definition \ref{independent-galois-type-amalgamation}. In particular, we have an independent horn in $\cc$ like the solid arrows below on the left, which we can complete with the dashed arrows using $3$-amalgamation.
\[
\vcenter{\vbox{\xymatrix@=1pc{
  & N_1 \ar@{-->}[rr] & & N' \\
  A \ar[ur] \ar[rr] & & N_2 \ar@{-->}[ur] & \\
  & B \ar[uu] \ar[rr] & & N_3 \ar@{-->}[uu] \\
  M \ar[uu] \ar[rr] \ar[ur] & & C \ar[uu] \ar[ur] &
}}}
\quad
\quad
\quad
\vcenter{\vbox{\xymatrix@=1pc{
  & F(N') \ar@{-->}[r] & N'' \\
  F(A) \ar[ur] \ar[rr] & & N \ar@{-->}[u] \\
  F(M) \ar[u] \ar[r] & F(N_3) \ar[uu] \ar[ur] &
}}}
\]
In particular, we have two independent squares like the solid arrows above on the right. By uniqueness for $\nf$ we thus find the dashed arrows making the diagram above on the right commute. Now, using the assumption that $F$ admits $1$-completions we find $N'' \to F(N^*)$ and a morphism $f: N' \to N^*$ in $\cc$ such that $F(f)$ is the composite $F(N') \to N'' \to F(N^*)$. Take $a^*$ to be $A \to N' \to N^*$ and $n_3$ to be $N_3 \to N' \to N^*$, then these form the required $\nf$-indpendent horn amalgamation.
\end{proof}

\section{Lifting base monotonicity}
\label{sec:lifting-base-monotonicity}
\begin{theo}\label{lift-base-monotonicity}
Let $\cc$ and $\cd$ be categories whose morphisms are monomorphisms that both have binary joins of subobjects, which are preserved by $F: \cc \to \cd$. If $\nf$ is an independence relation on $\cd$ satisfying base montonicity, monotonicity and invariance then $F^{-1}(\nf)$ satisfies base monotonicity.
\end{theo}
\begin{proof}
Suppose we are given an $F^{-1}(\nf)$-independent square like below on the left. Then $F(A) \nf^{F(M)}_{F(C)} F(B)$ and so $F(A) \vee F(B') \nf^{F(M)}_{F(B')} F(B)$ by Lemma \ref{base-monotonicity-with-joins}. We have $F(A) \vee F(B') = F(A \vee B')$, and so $F(A \vee B') \nf^{F(M)}_{F(B')} F(B)$. We thus conclude that $A \vee B' \nf^M_{B'} B$, yielding the diagram below on the right which proves base monotonicity.
\[
\vcenter{\vbox{\xymatrix@=2pc{
A \ar[rr] & & M \\
C \ar[u] \ar[r] \ar@{}[urr]|-{F^{-1}(\nf)} & B' \ar[r] & B \ar[u]
}}}
\implies
\vcenter{\vbox{\xymatrix@=2pc{
  F(A) \vee F(B') \ar[r] & F(M) \\
  F(B') \ar[r] \ar[u] \ar@{}[ur]|-{\nf} & F(B) \ar[u]
}}}
\implies
\vcenter{\vbox{\xymatrix@=2pc{
A \ar[r] & A \vee B' \ar[r] & M \\
C \ar[u] \ar[r] & B' \ar[r] \ar[u] \ar@{}[ur]|-{F^{-1}(\nf)} & B \ar[u]
}}}
\]
\end{proof}
\begin{exams}\label{motivating-examples-base-monotonicity}
We consider our motivating examples of bilinear spaces and exponential fields (Examples \ref{motivating-examples}).
\begin{enumerate}
\item The join of two bilinear subspaces $A, B \subseteq V$ is simply given by their span. So the forgetful functor $F: \Bil_K \to \Vec_K$ preserves finite joins of subobjects and hence lifts base monotonicity.
\item The canonical independence relation $\ind^{EA}$ on $\EAF$, which is the lift of $\ind^{\td}$ along $F: \EAF \to \ACF$ is known to not satisfy base monotonicity, so we know that $F$ cannot preserve binary joins of subobjects. The failure of base monotonicity is spelled out in detail in \cite[Example 4.23]{AHKK}. We will adjust that example to show how preserving binary joins fails.

Let $C$ be any EA-field and let $K$ be an algebraically closed field containing $C$ as well as elements $a, b, d$ that are algebraically independent over $C$. Set $A = C(a)^{\alg}$ and $B = C(b)^{\alg}$, where $(-)^{\alg}$ is the algebraic closure, and choose any exponential maps on them extending the one from $C$. Pick any $u \in C(a,b)^{\alg}$ that is not in $A + B$ (e.g., $u = ab$). We can then extend those exponential maps from $A$ and $B$ to all of $K$ such that $\exp(u) = d$. We now have that the join of $A$ and $B$ as algebraically closed subfields of $K$ is $C(a,b)^{\alg}$, whereas their join as EA-subfields is $\langle A \cup B \rangle^{EA}_K$, which contains $d$ and is thus bigger than $C(a,b)^{\alg}$. We conclude that $F$ does not preserve binary joins.
\end{enumerate}
\end{exams}
We finish this section by discussing a way to compute joins of subobjects in categories of the form $\cc_\cm$, and consequently how these are preserved by left multiadjoints. For this we will use multipushouts. These were implicitly defined earlier when we discussed multicolimits (see before Corollary \ref{multi}), but as we will actually use them here we recall their precise definition.
\begin{defi}
\label{multipushout}
Let $A \xleftarrow{f} C \xrightarrow{g} B$ be a span of morphisms. A \emph{multipushout} of this span consists of a set $\{A \xrightarrow{h_i} P_i \xleftarrow{k_i} B\}_{i \in I}$ of cocones, such that for any cocone $A \xrightarrow{a} D \xleftarrow{b} B$ there is a unique $i \in I$ and a unique $u: P_i \to D$ such that the following diagram commutes
\[
\xymatrix@=2pc{
A \ar[r]^{h_i} \ar@/^1.5pc/[rr]^a & P_i \ar[r]^{u} & D \\
C \ar[u]^{f} \ar[r]_{g} & B \ar[u]_{k_i} \ar@/_1pc/[ur]_{b} &
}
\]
We refer to each of the cocones in the set $\{A \xrightarrow{h_i} P_i \xleftarrow{k_i} B\}_{i \in I}$ as an \emph{instance} of the multipushout.
\end{defi}
We adjust the well-known construction of joins of subobjects using pushouts to one using multipushouts, and relative to a factorisation system.
\begin{lemma}\label{computing-joins}
Let $\cc$ be a category with a factorisation system $(\ce, \cm)$, where $\cm$ consists of monomorphisms. Suppose we are given a commuting square
\[
\xymatrix@=2pc{
A \ar@{^{(}->}[r]^a & D \\
C \ar[u] \ar[r] & B \ar@{_{(}->}[u]_b
}
\]
with $a, b \in \cm$ such that the multipushout of $A \leftarrow C \to B$ exists. Then we can compute the join $A \vee B$ of the subobjects $A, B \leq D$ in $\cc_\cm$ as follows. First, we let $P$ be the instance of the multipushout that admits a morphism into $D$. Then $A \vee B$ is the $(\ce, \cm)$-factorisation of this morphism $P \to D$:
\[
\xymatrix@=2pc{
A \ar[r] \ar@{^{(}->}@/^1.5pc/[rrr]^a & P \ar@{->>}[r] & A \vee B \ar@{^{(}->}[r] & D \\
C \ar[u] \ar[r] & B \ar[u] \ar@{^{(}->}@/_1pc/[urr]_b & &
}
\]
\end{lemma}
\begin{proof}
Since $\cm$ is a composable and left-cancellable class of monomorphisms
(see \cite[Proposition 2.1.1]{FK}), the morphism $A \to P \to A \vee B$ is in $\cm$, and so $A \leq A \vee B$ as subobjects in $\cc_\cm$. Similarly $B \leq A \vee B$. Now let $E \leq D$ be a subobject in $\cc_\cm$ such that $A, B \leq E$. Then we have the following diagram in $\cc$:
\[
\xymatrix@=2pc{
A \ar@{^{(}->}[r] \ar@{^{(}->}@/^1.5pc/[rr]^a & E \ar@{^{(}->}[r] & D \\
C \ar[u] \ar[r] & B \ar@{_{(}->}[u] \ar[u] \ar@{^{(}->}@/_1pc/[ur]_b &
}
\]
Any path from $C$ to $D$ gives the same morphism, and so the square involving $C, A, B, E$ commutes because $E \to D$ is in $\cm$ and is thus a monomorphism. Let $P'$ be the instance of the multipushout of $A \leftarrow C \to B$ admitting a morphism into $E$. Composing this morphism $P' \to E$ with $E \to D$ gives a morphism into the cocone $A \xrightarrow{a} D \xleftarrow{b} B$, and so $P' = P$.

So the morphism $P \to D$ factors through $E$, and as $E \to D$ is in $\cm$ with $(\ce, \cm)$ being a factorisation system, we have that $A \vee B \to D$ factors through $E$. By left-cancellability of $\cm$, this factorisation $A \vee B \to E$ must be in $\cm$, establishing that $A \vee B \leq E$ as subobjects of $D$ in $\cc_\cm$.
\end{proof}

\begin{coro}\label{left-multiadjoint-lifts-base-monotonicity}
Let $F: \cc \to \cd$ be a left multiadjoint and let $(\ce, \cm)$ be a factorisation system on $\cd$ with $\cm$ consisting of monomorphisms. Suppose that:
\begin{enumerate}
\item $\cc$ and $\cd$ have multipushouts;
\item $(F^{-1}(\ce), F^{-1}(\cm))$ is a factorisation system on $\cc$;
\item $F^{-1}(\cm)$ consists of monomorphisms;
\item $\cc$ and $\cd$ satisfy the right Ore condition with respect to cospans in $F^{-1}(\cm)$ (resp.\ $\cm$), so for any cospan $A \to D \leftarrow B$ of morphisms in $F^{-1}(\cm)$ (resp.\ $\cm$) there is a span $A \leftarrow C \to B$ (not necessarily in $F^{-1}(\cm)$ or $\cm$) making the relevant square commute.
\end{enumerate}
Then $\cc_{F^{-1}(\cm)}$ and $\cd_\cm$ have binary joins of subobjects and the restriction $F: \cc_{F^{-1}(\cm)} \to \cd_\cm$ preserves these joins.

In particular, if $\nf$ is an independence relation on $\cd_\cm$ satisfying base monotonicity, monotonicity and invariance then $F^{-1}(\ind)$ satisfies base monotonicity on $\cc_{F^{-1}(\cm)}$.
\end{coro}
\begin{proof}
The fact that $\cc_{F^{-1}(\cm)}$ and $\cd_\cm$ have binary joins follows from Lemma \ref{computing-joins}. Here we use assumption (4) to find the $C$ in that lemma for any cospan $A \to D \leftarrow B$ of morphisms in $F^{-1}(\cm)$ or $\cm$.

The fact that $F: \cc_{F^{-1}(\cm)} \to \cd_\cm$ preserves these binary joins then follows because $F$ preserves all the parts in the computation of a binary join. That is, $F$ sends a $(F^{-1}(\ce), F^{-1}(\cm))$-factorisation to a $(\ce, \cm)$-factorisation by definition, and left multiadjoints preserve connected multicolimits (see Corollary \ref{multi}).

The final claim follows immediately from Theorem \ref{lift-base-monotonicity}.
\end{proof}
The statement of Corollary \ref{left-multiadjoint-lifts-base-monotonicity} can be simplified by assuming $\cc$ to be locally multipresentable, which it often is in the intended applications. Then conditions (1) and (4) are automatic, as such categories are multicocomplete and have pullbacks. See also the example below.
\begin{exam}\label{lifting-base-monotonicity-difference-graphs}
Similarly to Example \ref{left-adjoint-examples}(2), we can consider the forgetful functor $U: \sigma\text{-}\Gra \to \Gra$, from the category of graphs with endomorphisms to the category of graphs, which are both locally presentable. As in \ref{left-adjoint-examples}(2), this functor is faithful and left adjoint. Consider the factorisation system $(\ce, \cm)$ = (surjections, graph em\-be\-ddings) in $\Gra$, and note that its pre-image under $U$ gives the same factorisation system in $\sigma\text{-}\Gra$. By Corollary \ref{left-multiadjoint-lifts-base-monotonicity} we thus have that the restriction $U: \sigma\text{-}\Gra_{\cm} \to \Gra_{\cm}$ preserves binary joins of subobjects.

We also note that the restriction $U: \sigma\text{-}\Gra_{\cm} \to \Gra_{\cm}$ does not admit $2$-completions, where the relevant independence relation on $\Gra_{\cm}$ will be the pullback squares (see also Example \ref{graphs-independent-horn-amalgamation}). Let $A = \{a_1, a_2\}$ be the $\sigma$-graph where the endomorphism swaps $a_1$ and $a_2$, and let $B = \{b_1, b_2\}$ be defined analogously. Let $D = \{a_1, a_2, b_1, b_2\}$, with an edge between $a_1$ and $b_1$. Then $A \ind^D_\emptyset B$ in $\Gra_{\cm}$. If $U$ were to admit $2$-completions then we would get a commuting diagram of graph embeddings like below
\[
\xymatrix@=2pc{
  U(A) \ar[r] \ar@{-->}@/^1.5pc/[rr] & D \ar@{-->}[r] & U(E) \\
  U(\emptyset) \ar[u] \ar[r] \ar@{}[ur]|-{\nf} & U(B) \ar[u] \ar@{-->}@/_1pc/[ur] &
}
\]
However, the endomorphism on $E$ would then have to send $a_1$ and $b_1$ to $a_2$ and $b_2$ respectively, which is impossible as there is an edge between $a_1$ and $b_1$ but not between $a_2$ and $b_2$.
\end{exam}

\section{Lifting everything}
\label{sec:lifting-everything}
\begin{defi}\label{reduct-functor}
A functor $F: \cc \to \cd$ between AECats is called a \emph{reduct functor} if it is faithful and preserves directed colimits.
\end{defi}
The intuition behind a reduct functor should be that of taking reducts, so it is some sort of forgetful functor. The forgetful functors in our motivating Examples \ref{motivating-examples} are of this form. More generally, when $\cc$ is a category of models of some theory $T$ in a signature $\cl$, then for any signature $\cl' \subseteq \cl$ and for any theory $T'$ in $\cl'$ with $T \models T'$, taking $\cl'$-reducts yields a reduct functor $F: \Mod(T) \to \Mod(T')$.
\begin{theo}\label{lift-everything}
Suppose that $F: \cc \to \cd$ is a reduct functor that admits $2$-completions. Let $\nf$ be an independence relation on $\cd$.
\begin{enumerate}
\item If $\nf$ is stable and $F$ reflects amalgamation of squares then $F^{-1}(\nf)$ is stable.
\item If $\nf$ is simple, $\cc$ and $\cd$ have binary joins of subobjects and $F$ preserves those joins and either $F$ has $\nf$-independent horn amalgamation or $F$ admits $3$-completions then $F^{-1}(\nf)$ is simple.
\item If $\nf$ is NSOP$_1$-like and $F$ has $\nf$-independent horn amalgamation then $F^{-1}(\nf)$ is NSOP$_1$-like.
\item If $\nf$ is NSOP$_1$-like with strong $3$-amalgamation and $F$ admits $3$-completions then $F^{-1}(\nf)$ is NSOP$_1$-like with strong $3$-amalgamation.
\end{enumerate}
\end{theo}
\begin{proof}
We put results together about which properties lift under which assumptions. The properties invariance, monotonicity, transitivity and symmetry all lift along any functor, by Proposition \ref{lifting-basic-properties}. Being a reduct functor allows us to apply Theorem \ref{lifting-union-accessibility} and conclude that the properties union and accessibility lift. The existence property lifts by Theorem \ref{admitting-2-completions-lifts-existence}, because $F$ admits $2$-completions.

Then for stable independence we only need to lift uniqueness, so we apply Proposition \ref{reflecting-amalgamation-of-squares-lifts-uniqueness}.

For simple independence we need to lift both $3$-amalgamation and base monotonicity. The latter follows from Theorem \ref{lift-base-monotonicity}. The former follows from Theorem \ref{independent-amalgamation-and-uniqueness} if $F$ has $\nf$-independent horn amalgamation. If $F$ admits $3$-completions then we note that by Theorem \ref{strong-3-amal-from-3-amal} $\nf$ satisfies strong $3$-amalgamation, and so we can apply Theorem \ref{admitting-3-completions-lifts-strong-3-amalgamation}.

Finally, for NSOP$_1$-like independence we only need to lift $3$-amalgamation. As before we either apply Theorem \ref{independent-amalgamation-and-uniqueness}, for case (3), or Theorem \ref{admitting-3-completions-lifts-strong-3-amalgamation}, for case (4).
\end{proof}
In the above proof for case (4) we needed to assume strong $3$-amalgamation, as opposed to case (2). This is because Theorem \ref{strong-3-amal-from-3-amal} no longer applies as it requires base monotonicity, which is exactly the differentiating property between 
simple and NSOP$_1$-like independence.
\begin{exam}
Theorem \ref{lift-everything} applies to both forgetful functors $\Bil_K \to \Vec_K$ and $\EAF \to \ACF$ from Examples \ref{motivating-examples}. For the former, case (2) applies and for the latter both (3) and (4) apply (strong $3$-amalgamation is implicit in \cite[Theorem 6.5]{HK}).
\end{exam}
\begin{theo}\label{left-multiadjoint-lifts-everything}
Let $F: \cc \to \cd$ be a faithful left multiadjoint and let $\cm$ be a left-cancellable composable accessible and continuous class of monomorphisms in $\cd$. Suppose that $\nf$ is an independence relation on $\cd_\cm$, that satisfies semi-invariance as an independence relation on $\cd$.
\begin{enumerate}
\item If $\nf$ is stable then $F^{-1}(\nf)$ is stable.
\item If $\nf$ is simple, $\cc_{F^{-1}(\cm)}$ and $\cd_\cm$ have binary joins of subobjects and $F$ preserves those then $F^{-1}(\nf)$ is simple.
\item If $\nf$ is NSOP$_1$-like then $F^{-1}(\nf)$ is NSOP$_1$-like, and furthermore if $\nf$ satisfies strong $3$-amalgamation then so does $F^{-1}(\nf)$.
\end{enumerate}
\end{theo}
\begin{proof}
We put results together about which properties lift under which assumptions. The properties invariance, monotonicity, transitivity and symmetry all lift along any functor, by Proposition \ref{lifting-basic-properties}. Left multiadjoints preserve connected colimits (Corollary \ref{multi}), and so in particular $F: \cc \to \cd$ preserves directed colimits. As $\cm$ is continuous, this implies that $F^{-1}(\cm)$ is continuous and so $F: \cc_{F^{-1}(\cm)} \to \cd_\cm$ preserves directed colimits. We can thus apply Theorem \ref{lifting-union-accessibility} to this restricted functor to see that union and accessibility are lifted. Existence is lifted by Theorem \ref{left-multiadjoint-lifts-existence}.

For stability, uniqueness is lifted by Theorem \ref{left-multiadjoint-lifting}. Then for both simplicity and NSOP$_1$-like independence, (strong) $3$-amalgamation is lifted by Theorem \ref{left-multiadjoint-lifts-3-amalgamation}.

That leaves us to show that base monotonicity lifts in case (2). For that, we wish to apply Theorem \ref{lift-base-monotonicity}, which requires $F^{-1}(\cm)$ to consist of monomorphisms. This follows quickly from the fact that $\cm$ consists of monomorphisms and that $F$ is faithful.
\end{proof}

\begin{theo}\label{left-multiadjoint-lifts-simplicity}
Let $F: \cc \to \cd$ be a faithful left multiadjoint between locally multipresentable categories, and let $(\ce, \cm)$ be a factorisation system on $\cd$ with $\cm$ an accessible and continuous class of monomorphisms. Suppose that $(F^{-1}(\ce), F^{-1}(\cm))$ is a factorisation system on $\cc$. If $\nf$ is a simple independence relation on $\cd_\cm$, satisfying semi-invariance as an independence relation on $\cd$, then $F^{-1}(\nf)$ is a simple independence relation on $\cc_{F^{-1}(\cm)}$.
\end{theo}
\begin{proof}
We follow the proof of Theorem \ref{left-multiadjoint-lifts-everything}, except we apply Corollary \ref{left-multiadjoint-lifts-base-monotonicity} instead to lift base monotonicity. This applies because assumptions (1) and (4) there are automatic for locally multipresentable categories and we assumed (2) to hold. Assumption (3), namely that $F^{-1}(\cm)$ consists of monomorphisms, was also shown to hold in the proof of Theorem \ref{left-multiadjoint-lifts-everything}.
\end{proof}
\begin{exam}
\label{connected-graphs}
We give an example, similar to Example \ref{left-multiadjoint-examples}(2), that uses all the components of Theorem \ref{left-multiadjoint-lifts-simplicity}. Let $\cc$ be the category of connected graphs with graph homomorphisms. Let $\cd = \Gra$ and let $(\ce, \cm) = (\text{surjections}, \text{embeddings})$. Then $F$ is a left multiadjoint. For a graph $G$, the family of connected components and their inclusions into $G$ witness that $F$ is left multiadjoint. Furthermore, $\cc$ is easily seen to be accessible and have all connected limits, so $\cc$ is locally multipresentable. Furthermore, $(F^{-1}(\ce), F^{-1}(\cm)) = (\text{surjections}, \text{embeddings})$. Let $\nf$ be the independence relation on $\Gra_{\Emb}$ given by pullback squares. Then $\nf$ is a semi-invariant simple independence relation (semi-invariance is quickly verified and simplicity can for example be found in \cite[Theorem 1.2 and Example 4.11]{KR}). So Theorem \ref{left-multiadjoint-lifts-simplicity} applies and we conclude that the pullback squares form a simple independence relation on $\cc_{\Emb}$. Finally, we note that $F$ is not a left adjoint for the same reasons as in Example \ref{left-multiadjoint-examples}(2).
\end{exam}
\begin{exam}\label{link-to-previous-work}
We now discuss a link to related work in \cite{KR}, where a category-theoretic construction of independence relations is given. We recall a simplified version of the main results in \cite{KR}, where we focus on regular monomorphisms rather than an arbitrary class $\cm$ of monomorphisms, but the current work applies to that generality as well.

Suppose that $\cc$ is a locally finitely presentable category and that the class $\Reg$ of regular monomorphisms is closed under composition. Then \cite[Theorem 1.1]{KR} states that if $\Reg$ is cubic (see \cite[Definition 3.8]{KR}) then the pullback squares form an NSOP$_1$-like independence relation on $\cc_{\Reg}$. Further conditions are given for when this independence relation is simple (see \cite[Theorem 1.2]{KR}) or stable (see \cite[Corollary 4.9]{KR} and \cite[Theorem 3.1]{LRV1}).

Let $M$ be any monoid and consider the category $\cc^M$ of objects in $\cc$ with an $M$-action, which is locally presentable. There is the obvious forgetful functor $U: \cc^M \to \cc$. This functor is faithful and preserves both limits and colimits, so it has both adjoints (by \cite[Theorems 1.58 and 1.66]{AR} and the Special Adjoint Functor Theorem). It also reflects limits (in fact, it is monadic), so $U^{-1}(\Reg)$ is the class of regular monomorphisms in $\cc^M$. Furthermore, regular monomorphisms are always continuous and accessible (see the proof of \cite[Lemma 4.3]{LRV}). They are are also part of a factorisation system $(\ce, \Reg)$, because they are composable (in essence this is \cite{Kelly}, and \cite[Fact 4.2]{KR} gives directions on how to put it in the present terms). Here $\ce$ is the class of epimorphisms, which $U$ preserves and reflects. In conclusion, Theorems \ref{left-multiadjoint-lifts-everything} and \ref{left-multiadjoint-lifts-simplicity} apply to $U$. The special case
$\Set^M$ is considered in \cite{MR}.

So if $\Reg$ is cubic then the pullback squares form an NSOP$_1$-like independence relation on $\cc_{\Reg}$, which is easily checked to satisfy semi-invariance as an independence relation on $\cc$. Then $U^{-1}(\ind)$, i.e.\ the pullback squares in $\cc^M_{\Reg}$, form an NSOP$_1$-like independence relation on $\cc^M_{\Reg}$. Furthermore, if $\ind$ is simple or stable then $U^{-1}(\ind)$ is simple or stable respectively. Note that Example \ref{left-adjoint-examples}(2) is a special case of this.
\end{exam}

\noindent \textbf{Acknowledgements.} We thank the anonymous referees for their feedback, which has greatly helped to improve the presentation of this paper.\\


\begin{wrapfigure}{r}{0.4\textwidth}
\includegraphics[height=1.1cm]{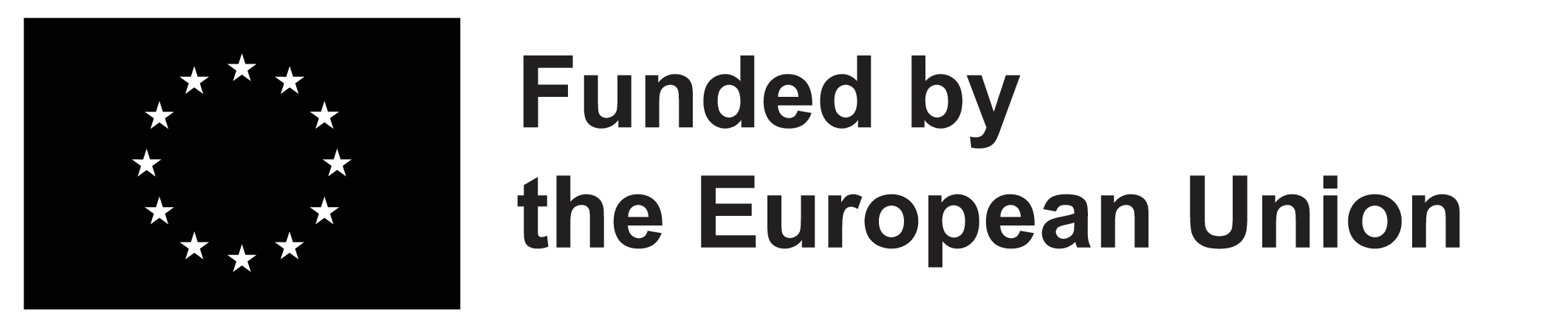}
\end{wrapfigure}
\noindent \textbf{Funding.} The first author is supported by the EPSRC grant EP/X018997/1 and Marie Sk\l{}odowska-Curie grant number 101130801 and the second author is supported by the Grant Agency of the Czech Republic under the grant 22-02964S.\\



\end{document}